\documentclass[a4paper,12pt,leqno]{article}
\usepackage[latin1]{inputenc} 
\usepackage[T1]{fontenc}   
\usepackage[french, english]{babel}  
\usepackage{typearea}
\usepackage{color}

\usepackage{amstext,amsthm,a4,amscd,amssymb}

\usepackage{textcomp}
\usepackage{amsmath}
\usepackage{amsfonts}
\usepackage{amssymb}
\usepackage{tikz}
\usepackage{fancyhdr}
\usepackage{url}
\usepackage{mathrsfs}
\usepackage{dsfont}
\usepackage{cancel}
\usepackage[justification=centering]{caption}

\usepackage
{hyperref}
\hypersetup{
    colorlinks = 
    true,
    linkcolor={black},
    linkbordercolor = {white},
    citecolor={black},
     }

\usepackage[all]{xy}

\usepackage[a4paper,top=2cm,bottom=2.5cm,left=2cm,
right=2cm,footskip=1cm]
{geometry}

\usepackage{tikz}
 \usetikzlibrary{patterns}
 
 \numberwithin{equation}{section}

\newtheorem{df}{Definition}[section]

\newtheorem{thm}[df]{Theorem}

\newtheorem{lem}[df]{Lemma}

\newtheorem{rmk}[df]{Remark}

\newcommand{\R}{\mathbb{R}}

\newcommand{\C}{\mathbb{C}}

\newcommand{\N}{\mathbb{N}}
\newcommand{\D}{\mathbb{D}}
\newcommand{\E}{\mathbb{E}}
\newcommand{\cO}{\mathscr{O}}

\newcommand{\bb}{\boldsymbol{b}}
\newcommand{\bL}{\boldsymbol{L}}

\newcommand*{\LargerCdot}{\raisebox{-0.25ex}
{\scalebox{1.5}{$\cdot$}}}
  
\newcommand{\cqfd}{ \hfill $\square$ }

\newcommand{\comment}[1]{}

\title{Quotient of Bergman kernels on punctured Riemann surfaces}
\author{\normalsize\textsc{Hugues AUVRAY\footnote{
Laboratoire de Math\'{e}matiques d'Orsay, Universit\'{e} Paris-Sud, 
CNRS, Universit\'{e} Paris-Saclay, D\'epartement de Math\'ematiques, 
B\^atiment 307, 91405 Orsay, France. 
E-mail: hugues.auvray@math.u-psud.fr}
\,\,and 
Xiaonan MA\footnote{Universit\'e de Paris, CNRS,
Institut de Math\'ematiques de Jussieu-Paris Rive Gauche,
F-75013 Paris, France. 
E-mail: xiaonan.ma@imj-prg.fr} 
\,\,and George MARINESCU\footnote{
Universit{\"a}t zu K{\"o}ln,  
Mathematisches Institut, Weyertal 86-90, 50931 K{\"o}ln, Germany. 
E-mail: gmarines@math.uni-koeln.de}}}
\date{}

\makeatletter
\def\blfootnote{\xdef\@thefnmark{}\@footnotetext}
\makeatother

\begin{document} 
\blfootnote{
First author is partially supported by ANR contract 
ANR-14-CE25-0010. 
Second author is partially supported by NNSFC No.\ 11829102
and funded through the Institutional Strategy of the University of Cologne
within the German Excellence Initiative. 
Third author is partially supported by CRC TRR 191.
            \vspace{4pt}
            \footnoterule
            \vspace{1pt}}
\makeatletter
\renewcommand\section{\@startsection {section}{1}{\z@}%
                                   {-3.5ex \@plus -1ex \@minus -.2ex}%
                                   {2.3ex \@plus.2ex}%
                                   {\centering\sc\normalsize}}
\renewcommand\subsection{\@startsection{subsection}{2}{\z@}%
                                     {-3.25ex\@plus -1ex \@minus -.2ex}%
                                     {1.5ex \@plus .2ex}%
                                     {\normalsize\sf}}
\renewcommand\subsubsection{\@startsection{subsubsection}{3}
{\z@}%
                                     {-3.25ex\@plus -1ex \@minus -.2ex}%
                                     {1.5ex \@plus .2ex}%
                                     {\normalsize\it}}

\makeatother
 
\maketitle
\abstract{In this paper we consider a punctured Riemann surface 
endowed with a Hermitian metric that equals the Poincar\'e metric near
the punctures, and a holomorphic line bundle that polarizes the metric.
We show that the quotient of the Bergman kernel of high tensor powers
of the line bundle and of the Bergman kernel of the Poincar\'e model near
the singularity tends to one up to arbitrary negative powers of 
the tensor power.
}

\selectlanguage{english}

\tableofcontents

\section{Introduction}

In this paper we study the asymptotics of Bergman kernels of high
tensor powers of a singular Hermitian 
line bundle over a Riemann surface under the assumption 
that the curvature has singularities of Poincar\'e type at a finite set. 
We show namely that the \emph{quotient} of these Bergman kernels
and of the Bergman kernel of the Poincar\'e model near
the singularity tends to one up to arbitrary negative powers of 
the tensor power.
In our previous paper \cite{bkp} (see also \cite{bkp0}) we obtained
a weighted estimate in the $C^m$-norm near the punctures for the 
\emph{difference} of the global Bergman kernel
and of the Bergman kernel of the Poincar\'e model near
the singularity, uniformly in the tensor powers of the given bundle.
Our method is inspired by the analytic localization technique of
Bismut-Lebeau \cite{BL}.

There exists a well-known expansion of the Bergman kernel on 
general compact manifolds \cite{bou, Ca99, DLM06, Hs10, mm, MM08,Ti90, Z98}
with important applications to the existence and uniqueness
of constant scalar curvature K\"ahler metrics \cite{Don,Ti90}
as part of the Tian-Yau-Donaldson's program.
Coming to our context,  a central problem is 
the relation between the existence of special complete/singular metrics
and the stability of the pair $(X, D)$ where $D$ is
a smooth divisor of a compact K\"{a}hler manifold $X$; 
see e.g.\ the suggestions of \cite[\S 3.1.2]{sze} for the case of 
\og asymptotically hyperbolic K\"ahler metrics\fg, 
which naturally generalize to higher dimensions the complete metrics 
$\omega_{\Sigma}$ studied here. 
Moreover, the technique developed 
here can be extended to the higher dimensional situation
in the case of Poincar\'e type K\"ahler metrics 
with reasonably fine asymptotics 
on complement of divisors, 
see the construction of \cite[\S1.1]{auv} and \cite[Theorem 4]{auv2}.

The Bergman kernel function of a singular polarization
is of particular interest in arithmetic situations \cite{BBK07, BKK05, bf}.
In \cite{bkp} we applied the precise asymptotics
of the Bergman kernel near the punctures in order to obtain optimal 
uniform estimates for the supremum of the Bergman kernel, relevant in 
arithmetic geometry \cite{AbUll95,jk04,fjk}. There are also applications to
\og partial Bergman kernels\fg, see \cite{CM15}.

  
We place ourselves in the setting of \cite{bkp}
which we describe now. 
Let $\overline\Sigma$ be a compact Riemann surface and let
$D=\{a_1,\ldots,a_N\}\subset\overline\Sigma$ be a finite set.
We consider the punctured Riemann surface 
$\Sigma = \overline{\Sigma}\smallsetminus D$ and a Hermitian form
$\omega_{\Sigma}$ on $\Sigma$.
Let $L$ be a holomorphic line bundle on $\overline{\Sigma}$, 
and let $h$ be a singular Hermitian metric on $L$ such that:
   \begin{itemize}
    \item[($\alpha$)] $h$ is smooth over $\Sigma$, 
and for all $j=1,\ldots,N$, there is a trivialization of $L$
 in the complex neighborhood $\overline{V_j}$ of $a_j$ in 
 $\overline{\Sigma}$, with associated coordinate $z_j$
 such that $|1|_{h}^2(z_{j})= \big|\!\log(|z_j|^2)\big|$.
    \item[($\beta$)] There exists $\varepsilon>0$ such that the
(smooth) curvature $R^L$ of $h$ satisfies
$iR^L\geq\varepsilon\omega_{\Sigma}$ over $\Sigma$ 
and moreover, $iR^L=\omega_{\Sigma}$ 
on $V_j:=\overline{V_j}\smallsetminus\{a_j\}$;
in particular, $\omega_{\Sigma} = \omega_{\D^*}$ 
in the local coordinate $z_j$ on $V_j$ 
and $(\Sigma, \omega_{\Sigma})$ is complete. 
\end{itemize}
Here $\omega_{\D^*}$ denotes the Poincar\'{e} metric on the 
punctured unit disc $\D^*$, normalized as follows:
\begin{equation}   \label{eqn_omegaPcr}
\omega_{\D^*} := \frac{idz\wedge d\overline{z}}
{|z|^2\log^2(|z|^2)}\,\cdot
\end{equation}
For $p\geq1$, let $h^p:=h^{\otimes p}$ be the metric induced 
 by $h$ on $L^p\vert_{\Sigma}$, 
where $L^p:=L^{\otimes p}$. We denote by 
$H^0_{(2)}(\Sigma,L^p)$ 
the space of ${\bL}^2$-holomorphic sections of $L^p$ 
relative to the metrics $h^p$ and $\omega_\Sigma$, 
\begin{equation}\label{e:bs}
H^0_{(2)}(\Sigma,L^p)=\left\{S\in H^0(\Sigma,L^p):\,
\|S\|_{{\bL}^2}^2:=\int_{\Sigma}|S|^2_{h^p}\,
\omega_\Sigma<\infty\right\},
\end{equation}
endowed with the obvious inner product. The sections from 
$H^0_{(2)}(\Sigma,L^p)$
extend to holomorphic sections of $L^p$ over $\overline\Sigma$, 
i.\,e.,  (see \cite[(6.2.17)]{mm})
\begin{equation}\label{e:bs1}
H^0_{(2)}(\Sigma,L^p)\subset 
H^0\big(\overline\Sigma,L^p\big).
\end{equation}
In particular, the dimension $d_p$ of $H^0_{(2)}(\Sigma,L^p)$ is finite.
 
We denote by $B_p(\LargerCdot,\LargerCdot)$ and 
by $B_p(\LargerCdot)$ 
the (Schwartz-)Bergman kernel and the Bergman kernel function
of the orthogonal projection $B_{p}$ from the space 
of $\bL^{2}$-sections of $L^{p}$ over $\Sigma$ onto
$H^0_{(2)}(\Sigma,L^p)$. They are defined as follows: 
if $\{S_\ell^p\}_{\ell=1}^{d_p}$ is an orthonormal 
basis of $H^0_{(2)}(\Sigma,L^p)$, then 
\begin{equation}\label{e:BFS1}
B_p(x,y):=\sum_{\ell=1}^{d_p}S^p_\ell(x)\otimes(S^p_\ell(y))^{*} 
\quad\text{and}\quad 
B_p(x):=\sum_{\ell=1}^{d_p}|S^p_\ell(x)|_{h^p}^2\,.
\end{equation} 
Note that these are independent of the choice of basis (see 
\cite[(6.1.10)]{mm} or \cite[Lemma 3.1]{CM11}). 
Similarly, let $B_p^{\D^*}(x,y)$ and $B_p^{\D^*}(x)$ be 
the Bergman kernel and Bergman kernel function of
$\big(\D^*, \omega_{\D^*}, 
\C,\big|\!\log(|z|^2)\big|^p\, h_{0})$ 
with $h_{0}$ the flat Hermitian metric on the trivial line bundle 
$\C$.
 
Note that for $k\in \N$, the $C^{k}$-norm at $x\in \Sigma$
is defined for $\sigma\in C^\infty(\Sigma, L^p)$ as 
\begin{equation}   \label{eq:2.13c}\begin{split}
 &|\sigma |_{C^k(h^p)}(x)= \big(  |\sigma|_{h^p}
 +\big|\nabla^{p,\Sigma}\sigma
\big|_{h^p,\omega_{\Sigma}}+\ldots+\big|(\nabla^{p,\Sigma})^k
\sigma\big|_{h^p,\omega_{\Sigma}}\big)(x),
\end{split}  \end{equation}
where $\nabla^{p,\Sigma}$ is the connection on
$(T\Sigma)^{\otimes\ell}\otimes L^p$ induced by the Levi-Civita 
connection on $(T\Sigma, \omega_{\Sigma})$ and the Chern connection
on $(L^{p},h^p)$,
and the pointwise norm $|\,\LargerCdot\,|_{h^p,\omega_{\Sigma}}$
is induced by $\omega_{\Sigma}$ and $h^{p}$. 
In the same way 
we define the $C^{k}$-norm $|f|_{C^{k}}(x)$ at $x\in \Sigma$ 
of a smooth function $f\in{C}^\infty(\Sigma,\C)$ by using 
the Levi-Civita connection on $(T\Sigma, \omega_{\Sigma})$.

We fix a point $a\in D$ and work 
in coordinates centered at $a$. 
Let $\mathfrak{e}_{L}$ be the holomorphic frame of $L$ near $a$ 
corresponding to the trivialization in the condition ($\alpha$).
By assumptions ($\alpha$) and ($\beta$) 
we have the following identification of the geometric data
in the coordinate $z$ on 
the punctured disc $\D^*_{4r}$ of radius $4r$ centered at $a$,
via the trivialization $\mathfrak{e}_{L}$ of $L$,
\begin{align}\label{eq:1.6a}
\big(\Sigma,\omega_{\Sigma}, L,h\big)\big|_{\D^*_{4r}}
=\big(\D^*,\omega_{\D^*}, \C, h_{\D^*}
= \big|\!\log(|z|^2)\big|\cdot h_{0}\big)\big|_{\D^*_{4r}}\,, 
\quad \text{ with } 0<r<(4e)^{-1}.
\end{align}
In \cite[Theorem 1.2]{bkp} we proved the following 
weighted diagonal expansion of the Bergman kernel:
\begin{thm}   \label{thm_MainThm}
Assume that $(\Sigma, \omega_{\Sigma}, L, h)$ fulfill conditions
($\alpha$) and ($\beta$). 
Then the following estimate holds: 
for any $\ell, k\in \N$, and every $\delta>0$, 
there exists $C=C(\ell, k, \delta)>0$ such that
for all $p\in\N^*$, and $z\in V_1\cup\cdots\cup V_N$
with the local coordinate $z_{j}$,
   \begin{equation}    \label{eqn_MainThm}
     \Big | B_p - B_p^{\D^*}\Big |_{C^k} (z_{j})\leq Cp^{-\ell}
     \, \big|\!\log(|z_{j}|^2)\big|^{-\delta},    
   \end{equation}
with norms computed with help of $\omega_{\Sigma}$ 
and the associated Levi-Civita connection on $\D_{4r}^*$.
    \end{thm}
\noindent    
Note that in \cite[Theorem 1.1]{bkp} we also established the
off-diagonal expansion of the Bergman kernel 
$B_p(\LargerCdot,\LargerCdot)$.
The main result of the present paper is the following
estimate of the quotient of the Bergman kernels from 
\eqref{eqn_MainThm}:
\begin{thm}   \label{thm_apdx}
If $(\Sigma, \omega_\Sigma, L, h)$ fulfill conditions $(\alpha)$ 
and $(\beta)$, then
 \begin{equation}   \label{e:apdx}
\sup_{z\in V_1\cup\ldots\cup V_N} 
\bigg|\frac{B_p}{B_p^{\D^*}}(z) - 1 \bigg| = 
\mathcal{O}(p^{-\infty})\,,
\end{equation}
i.e., for any $\ell>0$ there exists $C>0$ such that for any
$p\in\N^{*}$ we have
\begin{equation}   \label{e:apdx1}
\sup_{z\in V_1\cup\ldots\cup V_N} 
\bigg|\frac{B_p}{B_p^{\D^*}}(z) - 1 \bigg| \leq C p^{-\ell}.
\end{equation}
\end{thm}  
Theorem \ref{thm_apdx} is related to estimates 
in exponentially small neighborhoods of the punctures obtained 
in \cite[Theorem 1.6]{S} and \cite[Lemma 3.3]{SS}.

For each $p\geq 2$ fixed
$(|z|^{2}\big|\!\log(|z|^2)\big|^{p})^{-1} B_{p}^{\D^*}(z)$ is smooth 
and strictly positive on $\D_{4r}$, as follows from 
\eqref{eq:2.8a}. By \cite[Remark 3.2]{bkp},
any holomorphic $\bL^{2}$-section of $L^{p}$ over $\Sigma$
extends to a homomorphic section on $\overline{\Sigma}$ 
(see the inclusion \eqref{e:bs1}) vanishing at $0$ in $\D_{4r}$. 
Thus by the formula \eqref{e:BFS1} for $B_p$ we see that the quotient
$\frac{B_p}{B_p^{\D^*}}$ is a smooth 
function on $\D_{4r}$ for each $p\geq 2$.
  
\begin{thm}   \label{thm:diffquot}
For all $k\geq 1$ 
and $D_1,\ldots,D_k\in\Big\{\frac{\partial\,}{\partial z}\,,
\frac{\partial\,}{\partial \overline{z}}\Big\}$  we have
\begin{equation}   \label{e:diffquot}
\sup_{z\in\overline{V_1}\cup\ldots\cup \overline{V_N}}
\Big|(D_1\cdots D_k)\frac{B_p}{B_p^{\D^*}}(z)\Big| 
= \mathcal{O}(p^{-\infty}).
\end{equation} 
\end{thm}
 
\begin{rmk}
\rm{Theorem \ref{thm_MainThm} admits a generalization to
orbifold Riemann surfaces.   
Indeed, assume that $\overline\Sigma$ is a compact orbifold Riemann
surface such that the finite set 
$D\subset\overline\Sigma$
does not meet the (orbifold) singular set of $\overline\Sigma$.
Then by the same argument as in \cite[Remark 1.3]{bkp} (using
\cite{DLM06,DLM12}) we see that
Theorems \ref{thm_apdx} and \ref{thm:diffquot} still hold in this context.}
\end{rmk}	
 
Note that the $C^k$-norm used in \eqref{eqn_MainThm}
 is induced by $\omega_{\D^*}$, roughly the sup-norm with respect to
 the derivatives defined by the vector fields
$z\log(|z|^2) \frac{\partial\,}{\partial z}$ and
$\overline{z}\log(|z|^2)  \frac{\partial\,}{\partial \overline{z}}$, 
which vanish at $z=0$.
Hence the norm in \eqref{e:diffquot} is stronger than 
the $C^k$-norm used in \eqref{eqn_MainThm}, 
because the norm in \eqref{e:diffquot} is defined by 
using derivatives along the vector fields
$\frac{\partial\,}{\partial z}$ and 
$\frac{\partial\,}{\partial \overline{z}}$\,.
   
  Let us mention at this stage that even if the results above follow from 
  our work \cite{bkp}, relying more precisely on \cite[Theorem 1.2]{bkp}, 
  the proofs are by no means an obvious rewriting of 
  \cite[Theorem 1.2]{bkp} 
  (for instance), since $B_p^{\D^*}(\LargerCdot)$ takes extremely 
  small values arbitrarily near the origin. This can be seen in 
  \cite[\S3.2]{bkp} and it is specific to the non-compact framework. 
  What estimate \eqref{e:apdx} says is that $B_p(\LargerCdot)$ 
  follows such a behaviour very closely
  in the corresponding regions of $\Sigma$ via the chosen coordinates. 
  
  Here is a general strategy of our approach for 
  Theorems \ref{thm_apdx} and \ref{thm:diffquot}. We choose a special 
  orthonormal basis $\{\sigma^{(p)}_{\ell}\}_{\ell=1}^{d_{p}}$
  of $H^{0}_{(2)}(\Sigma, L^{p})$ starting from $z^{l}$ 
  on $\D^{*}_{4r}$ for 
  $1\leq l\leq \delta_{p}$ with $0<\alpha<\delta_{p}/p <\alpha_{1}<1$.
  Our choice of $\sigma^{(p)}_{\ell}$ implies that the coefficients of the 
  expansion 
  $$\sigma^{(p)}_{\ell}(z)=\sum_{j=1}^{\infty} a^{(p)}_{j\ell}z^{j}$$
  of $\sigma^{(p)}_{\ell}$ on $\D_{4r}^{*}$ satisfy
  $a^{(p)}_{j\ell}=0$ if $j<\delta_{p}$ and $j<l\leq d_{p}$ 
  (cf.\ \eqref{eq:2.26a}). Now we separate the contribution of 
  $\sigma^{(p)}_{\ell}$, $c^{(p)}_{\ell}$  (cf.\ \eqref{eq:2.7a}), 
  $a^{(p)}_{j\ell}$ in $B_{p}, B_{p}^{\D^{*}}$ in two groups: 
  $1\leq j, \ell\leq \delta_{p}$; $\max\{j, \ell\}\geq \delta_{p}+1$.
  The contribution corresponding to $1\leq j, \ell\leq \delta_{p}$,
  will be controlled by using Lemma \ref{lem_phi0phisigma} 
  (or \ref{prop:rfnd_estmt}). The contribution corresponding to 
  $\max\{j, \ell\}\geq \delta_{p}+1$ will be handled by a direct 
  application of Cauchy inequalities \eqref{eq:2.27a}.  It turns out that
  by suitably choosing $c,A>0$ this contribution has uniformly
  the relative size $2^{-\alpha p}$ compared to $B_{p}^{\D^{*}}$ on 
  $|z|\leq cp^{-A}$.

This paper is organized as follows. In Section \ref{eq:s2}, 
we establish Theorem \ref{thm_apdx} based on 
the off-diagonal expansion of Bergman kernel from \cite[\S 6]{bkp}.
In Section \ref{eq:s3},  we establish Theorem \ref{thm:diffquot}
by refining the argument from Section \ref{eq:s2}.
In Section \ref{eq:s4} we give some applications of the main results.
  
Notation:   We denote $\lfloor x\rfloor$ as the integer part of $x\in \R$.
  
\smallskip
\noindent
\textbf{\emph{Acknowledgments.}} We would like to thank Professor 
Jean-Michel Bismut for helpful discussions. In particular, 
Theorem \ref{thm:diffquot} answers a question raised 
by him at CIRM in October 2018.

\section{$C^{0}$-estimate for the quotient of Bergman kernels} 
\label{eq:s2}
  
This section is organized as follows. In Section \ref{eq:s2.1},
we obtain the $C^{0}$-estimate for the quotient of Bergman kernels,
Theorem \ref{thm_apdx}, admitting first an  integral estimate,
Lemma \ref{lem_phi0phisigma}. In Section  \ref{eq:s2.2},
we deduce Lemma \ref{lem_phi0phisigma} from 
the \textit{two-variable} Poincar\'{e} type Bergman kernel estimate 
of \cite[Theorem 1.1 and Corollary 6.1]{bkp}.

\subsection{Proof of Theorem \ref{thm_apdx}}\label{eq:s2.1}

We recall first some basic facts.
For $\sigma\in C_{0}^{\infty}(\Sigma, L^{p})$,
the space of smooth and compactly supported sections of $L^{p}$ over 
$\Sigma$, set
\begin{align}\label{eq:2.1a}
\|\sigma\|_{\bL^2_p(\Sigma)}^{2}
:= \int_{\Sigma}|\sigma|^{2}_{h^{p}}\, 	\omega_{\Sigma}.
\end{align}
Let $\bL^2_p(\Sigma)$ be the 
$\|\cdot\|_{\bL^{2}_p(\Sigma)}$-completion of 
$C_{0}^{\infty}(\Sigma, L^{p})$.

By \cite[Remark 3.2]{bkp} the inclusion \eqref{e:bs1} identifies the 
space $H^{0}_{(2)}(\Sigma, L^{p})$ of $\bL^2$-holomorphic sections 
of $L^p$ over $\Sigma$ to the subspace of $H^0(\overline\Sigma,L^p)$
consisting of sections vanishing at the punctures, so it induces
an isomorphism of vector spaces
\begin{equation}\label{e:bs2}
H^0_{(2)}(\Sigma,L^p)\cong
H^{0}(\overline{\Sigma},L^{p}\otimes
\mathscr{O}_{\overline{\Sigma}}(-D)),
\end{equation}
where $\mathscr{O}_{\overline{\Sigma}}(-D)$ is the holomorphic line 
bundle on $\overline{\Sigma}$ defined by the divisor $- D$. 
By the Riemann-Roch theorem we have for all $p$ with $p\deg(L)-N>2g-2$, 
\begin{align}\label{eq:2.4a}
d_p: = \dim H^{0}_{(2)}(\Sigma, L^{p})
= \dim H^{0}(\overline{\Sigma}, 
L^{p}\otimes \mathscr{O}_{\overline{\Sigma}}(-D))
= \deg (L)\, \,  p +1- g -N, 
\end{align}
where $\deg(L)$ is the degree of $L$ over 
$\overline{\Sigma}$, and $g$ is the genus of $\overline{\Sigma}$.

The Bergman kernel function \eqref{e:BFS1} satisfies the following
variational characterization, see e.g.\ \cite[Lemma 3.1]{CM11}, 
\begin{align}\label{eq:2.5a}
B_{p}(z) = \sup_{0\neq \sigma\in H^{0}_{(2)}(\Sigma, L^{p})}
\frac{|\sigma(z)|^{2}_{h^{p}}}{\|\sigma\|_{\bL^2_p(\Sigma)}^{2}}\,,
\quad \text{ for } z\in \Sigma.
\end{align}
By the expansion of the Bergman kernel
on a complete manifold \cite[Theorem 6.1.1]{mm}
(cf.\ also \cite[Theorem 2.1, Corollary 2.4]{bkp}), 
there exist coefficients 
$\bb_i\in C^\infty(\Sigma)$,  $i\in\N$, such that for any $k,m\in \N$,
any compact set $K\subset \Sigma$,
we have in the $C^{m}$-topology on $K$, 
\begin{align}\label{eq:2.6a}
B_p(x)=\sum^k_{i=0}\bb_i(x)p^{1-i}+\mathcal{O}(p^{-k})\,, 
\quad \text{ as } p\to \infty,
\end{align}
with $\bb_0=- \bb_1= \frac{1}{2\pi}$ on each $V_{j}$.

Consider now for $p\geq2$ the space $H_{(2)}^p(\D^*)$
of holomorphic ${\bL}^2$-functions on $\D^*$ with respect to the weight 
$\|1\|^{2}(z)=\big|\!\log(|z|^2)\big|^p$ 
(corresponding to a metric on the trivial line bundle $\C$)
and volume form $\omega_{\D^*}$ on $\D^*$.
An orthonormal basis of $H_{(2)}^p(\D^*)$ is given by 
(cf.\ \cite[Theorem 3.1]{bkp}), 
\begin{equation}\label{eq:2.7a}
  c^{(p)}_\ell z^\ell \text{ with }  \ell\in\N,\,\ell\geq 1 \text{ and } 
c^{(p)}_\ell = \left(\dfrac{\ell^{p-1}}{2\pi (p-2)!}\right)^{1/2}
   = \|z^{\ell}\|_{\bL^2_p(\D^*)}^{-1}\,,
\end{equation}
and hence
\begin{equation}\label{eq:2.8a}
B_p^{\D^*}(z) = \big|\!\log(|z|^2)\big|^{p}
\sum_{\ell=1}^{\infty} (c^{(p)}_\ell)^2 |z|^{2\ell}\,,
\qquad \text{ for }  z\in \D^{*}.
\end{equation}
For any $m\in \N$, $0<b<1$ 
and $0<\gamma<\frac{1}{2}$ 
there exists by \cite[Proposition 3.3]{bkp} 
$\epsilon = \epsilon(b,\gamma)>0$ such that
 \begin{align}\label{eq:2.12a}
   \Big\|B_p^{\D^*}(z) - \frac{p-1}{2\pi}
   \Big\|_{C^m(\{b e^{-p^\gamma}\leq |z|<1\}, \omega_{\D^*})} 
   = \mathcal{O}\big(e^{-\epsilon p^{1-2\gamma}}\big)
   \:\: \text{ as } p\to +\infty\, .
  \end{align} 
Taking into account Theorem \ref{thm_MainThm} and \eqref{eq:2.12a} 
we see that in order to prove Theorem \ref{thm_apdx} it suffices, 
after reducing to some $V_j$ 
and identifying the geometric data on $\D^*_{4r}$ and $\Sigma$ 
via \eqref{eq:1.6a},  
to show that for some (small) $c>0$ and (large) $A>0$, 
and for all $l\geq 0$ there exists $C=C(c,A,l)>0$ such that 
for all $p\geq 2$, 
    \begin{equation}   \label{e:apdx2}
      \sup_{0<|z|\leq cp^{-A}} 
	  \bigg|\frac{B_p}{B_p^{\D^*}}(z) - 1 \bigg| \leq Cp^{-l}. 
    \end{equation}
   %
We now start to establish \eqref{e:apdx2}. In the whole paper
we use the following conventions.
\begin{align}\label{eq:2.13a}\begin{split}
&\text{We fix $0<r<(4e)^{-1}$ as in \eqref{eq:1.6a}, 
 and $0<\beta<1$ such  that $r^{\beta}<2r$.} \\
&\text{We fix a (non-increasing) smooth cut-off function 
$\chi:[0,1]\to \R$, }\\
&\text{satisfying $\chi(u)=1$ if $u\leq r^{\beta}$
 and $\chi(u)=0$ 
 if $u\geq 2r$.} \\
& \text{We set}\;\;\delta_p = \left\lfloor \frac{p-2}{2|\log r|}\right\rfloor
 \text{ for } p\in \N, p\geq 2.
   \end{split} \end{align} 
The choice of $\delta_p$ will become clear in \eqref{e:implies},
\eqref{eq:2.31a} and \eqref{eq:3.6a}, for example.   
By \eqref{eq:2.13a} there exist $\alpha>0$ such that 
\begin{equation}   \label{e:alphabeta}
\alpha p\leq \delta_p \quad \text{ and }    \quad 
\delta_p +1 \leq \frac{1}{2}\,  p \quad \text{ for }
p \geq 2 + 2 |\log r|. 
\end{equation}
To establish \eqref{e:apdx2} we proceed along the following lines: 

\begin{enumerate}
\item 
For $\ell\in\{1,\ldots,\delta_p\}$, we set
\begin{equation}   \label{e:dfphi0}
 \phi^{(p)}_{\ell,0} = c^{(p)}_\ell\chi(|z|)z^\ell.
\end{equation}
     \item Using the trivialization, 
     that is, identifying $\phi^{(p)}_{\ell,0}$ with 
	$\phi^{(p)}_{\ell,0}\mathfrak{e}_L^p$ when we work on $\Sigma$, 
      we see the $\phi^{(p)}_{\ell,0}$ as (smooth) $\bL^2$ sections of 
	$L^p$ over $\Sigma$, that we correct into \textit{holomorphic} 
    $\bL^2$ sections $\phi^{(p)}_{\ell}$ of $L^p$, 
     by orthogonal $\bL^2_p(\Sigma)$-projection. 
  \item Next we correct the family 
$(\phi^{(p)}_{\ell})_{1\leq\ell\leq\delta_p}$ 
 into an \textit{orthonormal} family
	$(\sigma^{(p)}_{\ell})_{1\leq\ell\leq\delta_p}$ 
  by the Gram-Schmidt procedure, and we further complete 
  $(\sigma^{(p)}_{\ell})_{1\leq\ell\leq\delta_p}$
 into an orthonormal basis $(\sigma^{(p)}_{\ell})_{1\leq\ell\leq d_p}$ of 
 $H^{0}_{(2)}(\Sigma, L^{p})$. 
In particular, for any $1\leq j \leq \delta_{p}$,
\begin{equation}    \label{eq:2.14a}
{\rm Span}\big\{\phi^{(p)}_{1,0},\cdots, \phi^{(p)}_{j,0}\big\}
= {\rm Span}\big\{\phi^{(p)}_{1},\cdots, \phi^{(p)}_{j}\big\}
= {\rm Span}\big\{\sigma^{(p)}_{1},\cdots, \sigma^{(p)}_{j}\big\}.
\end{equation}
     \item Finally, we carefully compare $B^{\D^{*}}_{p}$ with $B_{p}$
using the three steps of the above construction to 
get estimate \eqref{e:apdx2}; 
of particular importance are the following intermediate estimates
which will be deduced from \cite[\S 6]{bkp}:
    \end{enumerate}
   \begin{lem}   \label{lem_phi0phisigma}
    With the notations above, 
    for all $m\in \N$, there exists  $C=C(m)>0$ 
    such that for all $p\in \N^{*}$, $p\geq 2$,
	and all $j,\ell\in\{1,\ldots,\delta_p\}$, 
     \begin{equation}   \label{e:phi0}
     \begin{split}
      1 - Cp^{-m}
\leq  \big\|\phi^{(p)}_{\ell,0}\big\|_{\bL^2_p(\Sigma)}^2&=
\big(c^{(p)}_\ell\big)^2 \int_{\D^*_{2r}} 
 \chi^{2}(|z|) |z|^{2\ell}\big|\!\log(|z|^2)\big|^{p}\, \omega_ {\D^*}\\
 &\leq \big(c^{(p)}_\ell\big)^2 \int_{\D^*_{2r}} \chi(|z|)|z|^{2\ell}\big|
	   \!\log(|z|^2)\big|^{p}\, \omega_ {\D^*}
        \leq 1\,,
        \end{split}
     \end{equation}
and moreover, 
     \begin{align}    \label{e:phi0phi} \begin{split}  
&\big\|\sigma^{(p)}_{\ell}- \phi^{(p)}_{\ell,0}\big\|_{\bL^2_p(\Sigma)}		 
	\leq C p^{-m},\\
   &   \big|\big\langle \phi^{(p)}_{j}, \sigma^{(p)}_{\ell}
	  \big\rangle_{\bL^{2}_p(\Sigma)} - \delta_{j\ell}\big| 
                    \leq C p^{-m}.
\end{split}     \end{align}
   \end{lem}
The proof of Lemma \ref{lem_phi0phisigma} is postponed 
to Section \ref{eq:s2.2}.

  Notice that we take care of stating estimates uniform in 
  $j,\ell\in \{1,\ldots,\delta_p\}$.
  Observe moreover that \eqref{e:phi0}, \eqref{e:phi0phi} 
  are   \textit{integral} estimates, 
  whereas we want to establish \textit{pointwise} estimates in the end, 
  hence we need an extra effort to convert these (among others)
  into \eqref{e:apdx2}. 
  
Let  us see now how to build on \eqref{e:phi0phi} 
to get  the desired \eqref{e:apdx2}.

  First, by  \eqref{e:BFS1},   \eqref{eq:2.8a}, \eqref{e:dfphi0},  
  and the construction of  $\phi_{\ell,0}^{(p)}$ and $\sigma^{(p)}_\ell$
  we have for $z \in \D^*_r$, 
   \begin{align}\label{eq:2.16a}\begin{split}
    B_p^{\D^*}(z) &
      \,   = \sum_{\ell=1}^{\delta_p} \big| \phi^{(p)}_{\ell,0}
	  \big|_{h^p,z}^2 
              + \big|\!\log(|z|^2)\big|^{p}\sum_{\ell=\delta_p+1}^{\infty} 
			  (c^{(p)}_\ell)^2 |z|^{2\ell}    \\
 \,   &= B_p(z) - \sum_{\ell=\delta_p+1}^{d_p}
 \big| \sigma^{(p)}_{\ell}\big|_{h^p,z}^2  
+2{\rm Re}\Big[\sum_{\ell=1}^{\delta_p} 
    \big\langle \sigma^{(p)}_{\ell}, 
  \phi^{(p)}_{\ell,0} -\sigma^{(p)}_{\ell}\big\rangle_{h^p,z}\Big]   \\
 & \quad   +\sum_{\ell=1}^{\delta_p} \big| \phi^{(p)}_{\ell,0} 
 -\sigma^{(p)}_{\ell}\big|_{h^p,z}^2
    +\big|\!\log(|z|^2)\big|^{p}\sum_{\ell=\delta_p+1}^{\infty} 
	(c^{(p)}_\ell)^2 |z|^{2\ell}.    
 \end{split}  \end{align}
  We deal with the summands 
  of the last three terms separately; 
  we start by claiming that up to a judicious choice of $c>0$ and $A>0$ 
  we have for $0<|z|\leq cp^{-A}$: 
   \begin{equation}   \label{e:btail}
    \big|\!\log(|z|^2)\big|^{p}\sum_{\ell=\delta_p+1}^{\infty}
	(c^{(p)}_\ell)^2 |z|^{2\ell}
    = \mathcal{O}(p^{-\infty})\cdot B_p^{\D^*}(z)\,,
    \qquad\text{as }p\to \infty,
   \end{equation}
  that is, 
  $$\sup_{0<|z|\leq cp^{-A}}
    \big[ B_p^{\D^*}(z)^{-1}|\!\log(|z|^2)|^{p}
	\sum_{\ell=\delta_p+1}^{\infty} (c^{(p)}_\ell)^2 |z|^{2\ell}\big]
   = \mathcal{O}(p^{-\infty}).$$
Indeed, we have $\frac{\ell+\delta_p}{\ell} \leq \delta_p+1$
for all $\ell\geq 1$, so by \eqref{eq:2.7a} we have for $z\in \D^*$, 
   \begin{equation} \label{eq:2.18a}
   \begin{split}
    \big|\!\log(|z|^2)\big|^{p}\sum_{\ell=\delta_p+1}^{\infty}
	(c^{(p)}_\ell)^2 |z|^{2\ell}
&=   \big|\!\log(|z|^2)\big|^{p}\frac{|z|^{2\delta_p}}{2\pi (p-2)!} 
    \sum_{\ell=1}^{\infty} \Big(\frac{\ell+\delta_p}{\ell}\Big)^{p-1} 
	\ell^{p-1}|z|^{2\ell}   \\      
&\leq (\delta_p+1)^{p-1}\frac{|z|^{2\delta_p}}{2\pi (p-2)!} 
             \big|\!\log(|z|^2)\big|^{p}\sum_{\ell=1}^{\infty}  
			 \ell^{p-1}|z|^{2\ell}            \\
&=(\delta_p+1\big)^{p-1}|z|^{2\delta_p}B_p^{\D^*}(z). 
   \end{split}
   \end{equation}
From \eqref{e:alphabeta} follows
   \begin{equation} \label{e:implies}
   ( \delta_p+ 1) |z|^{2\delta_p/(p-1)} 
   \leq \frac{1}{2} p |z|^{2\alpha}   \leq \frac{1}{2}\,, 
\quad   \text{ for all } |z|\leq p^{-1/(2\alpha)}.
   \end{equation}
From \eqref{eq:2.18a} and \eqref{e:implies} we get 
\eqref{e:btail} with $c=r $ and $A = \frac{1}{2\alpha}\,\cdot$ 

In the similar vein we now show the following.
\begin{lem}\label{eq:t2.2} 
For $c=r$ and $A = \frac{1}{2\alpha}$, with $\alpha$ 
satisfying \eqref{e:alphabeta}, we have
 uniformly in $z\in \D^*_{cp^{-A}}$, 
   \begin{equation}   \label{e:sumsigma}
    \sum_{\ell=\delta_p+1}^{d_p} \big| \sigma^{(p)}_{\ell}\big|_{h^p,z}^2 
     = \mathcal{O}(p^{-\infty})\cdot B_p^{\D^*}(z).
   \end{equation}
We have uniformly in $z\in \D^*_{cp^{-A}}$, and  
$\ell\in \{1,\ldots,\delta_p\}$, 
   \begin{equation}   \label{e:sigmaminusphi0}
    \big|\sigma^{(p)}_\ell - \phi^{(p)}_{\ell,0}\big|_{h^p,z} 
     = \mathcal{O}(p^{-\infty})\cdot B_p^{\D^*}(z)^{1/2} .
   \end{equation}
 \end{lem}  
\begin{proof}  
Let $p\geq 2$, $\ell\in \{1,\cdots,d_p\}$.
By \cite[Remark 3.2]{bkp} and \eqref{e:bs1} we know that
 $\sigma^{(p)}_{\ell}$ is a holomorphic section of $L^p$ over
 $\overline{\Sigma}$ vanishing at $D$.
 We use the trivialization \eqref{eq:1.6a} to set
\begin{align}   \label{eq:2.23a}
\sigma^{(p)}_{\ell} = \Big(\sum_{j=1}^{\infty}a^{(p)}_{j \ell} z^j\Big)
  \mathfrak{e}_L^p =: s^{(p)}_{\ell}\mathfrak{e}_L^p 
  \qquad \text{ on } \D^*_{4r}.
    \end{align}
We have for $j\geq 1$ by \eqref{e:BFS1}, \eqref{eq:1.6a},
\eqref{eq:2.6a}, \eqref{eq:2.23a} 
and Cauchy inequalities,
\begin{equation} \label{eq:2.27a}
\begin{split}
|a^{(p)}_{j \ell}|  &\leq (2r)^{-j} \sup_{|z|=2r}\big|s^{(p)}_{\ell}(z)\big|   \\
&= (2r)^{-j}\big|\!\log(|2r|^2)\big|^{-p/2} 
               \sup_{|z|=2r}\big|\sigma^{(p)}_{\ell}(z)\big|_{h^p}     \\
&\leq (2r)^{-j}\big|\!\log(|2r|^2)\big|^{-p/2}
                  \sup_{|z|=2r}B_p(z)^{1/2}        \\                           
     &\leq Cp^{1/2}(2r)^{-j}\big|\!\log(|2r|^2)\big|^{-p/2}.
     \end{split}
   \end{equation}
Thus by \eqref{eq:1.6a} and \eqref{eq:2.27a} we have for $z\in \D^*_{r}$,
\begin{equation} \label{eq:2.28a}
  \begin{split}
    \Big| \sum_{j=\delta_p+1}^{\infty} a^{(p)}_{j \ell}z^j\Big|_{h^p} 
&\leq C \big|\!\log(|z|^2)\big|^{p/2} \sum_{j=\delta_p+1}^{\infty} p^{1/2}
 \Big|\!\log(|2r|^2)\Big|^{-p/2}  \left(\frac{|z|}{2r}\right)^j   \\
  &=   Cp^{1/2}\left(\frac{|\!\log(|z|^2)|}{|\!\log(|2r|^2)|}\right)^{p/2}
 \left(1-\frac{|z|}{2r}\right)^{-1} \left(\frac{|z|}{2r}\right)^{\delta_p +1}.
      \end{split}
   \end{equation}
 By \eqref{eq:2.7a} and \eqref{eq:2.8a} we have 
    \begin{align} \label{eq:2.29a}
    \big|\!\log(|z|^2)\big|^{p/2}|z| 
    = |z|_{h^p_{\D^*}}\leq \|z\|_{\bL^2_p(\D^*)}B_p^{\D^*}(z)^{1/2} 
    = (2\pi(p-2)!)^{1/2}B_p^{\D^*}(z)^{1/2}.
      \end{align}
We deduce from \eqref{eq:2.28a} and \eqref{eq:2.29a} that 
there exists $C>0$ such that the following estimate holds uniformly in 
$\ell\in \{1,\ldots,d_p\}$, $|z|\in \D^*_{r}$, 
   \begin{equation} \label{eq:2.30a}
    \Big|\sum_{j=\delta_p+1}^{\infty} a^{(p)}_{j \ell} z^j\Big|_{h^p} 
      \leq  Cp^{-1/2} \left(  \left(\frac{|z|}{2r}\right)^{2\delta_p/p}
\frac{(p!)^{1/p}}{|\!\log(|2r|^2)|}\right)^{p/2}  B_p^{\D^*}(z)^{1/2}.
 \end{equation}
By \eqref{e:alphabeta} we have for $A=\frac{1}{2\alpha}$,
$c_0= r e^{1/(2\alpha)}|\!\log(|2r|^2)|^{1/(2\alpha)}> r$,  
and $p\gg 1$,
   \begin{align} \label{eq:2.31a}
\left(\frac{|z|}{2r}\right)^{2\delta_p/p}
     \frac{1}{|\!\log(|2r|^2)|}   \leq 
     \left(\frac{|z|}{2r}\right)^{2\alpha} \frac{1}{|\!\log(|2r|^2)|}  
     \leq 2^{- 2\alpha} \frac{e}{p}\,,\quad  \text{ for }  |z|\leq c_0 p^{-A}.
 \end{align}
Recall that the Stirling formula states 
   \begin{align} \label{eq:2.34a}
   \frac{p^p}{p!}= (2\pi p)^{-1/2} e^p \Big(1+\mathcal{O}(p^{-1})\Big)
\quad    \text{  as  } p\to +\infty.
  \end{align}  
We infer from \eqref{eq:2.30a}, \eqref{eq:2.31a} and \eqref{eq:2.34a}, 
that there exists $C>0$ such that the following estimate holds uniformly in 
 $|z|\leq r \, p^{-A}$ and $\ell\in \{1,\ldots,d_p\}$,
   \begin{align} \label{eq:2.32a}
       \Big|\sum_{j=\delta_p+1}^{\infty} a^{(p)}_{j \ell} z^j\Big|_{h^p} 
      \leq  C\,  2^{-p\alpha}\,   B_p^{\D^*}(z)^{1/2}. 
   \end{align}      
  Note that $ \phi_{j}^{(p)} - \phi_{j,0}^{(p)}$ is orthogonal to
  $H^0_{(2)}(\Sigma, L^p)$.
  By \eqref{eq:2.1a},  \eqref{e:dfphi0}, \eqref{eq:2.23a},
  and since $\sigma^{(p)}_{\ell}$ are holomorphic, we have for 
  $j\in\{1,\cdots,\delta_p\}$,  $\ell\in \{1,\cdots,d_p\}$, 
   \begin{align}   \label{e:aj}
 \big\langle \sigma^{(p)}_{\ell},\phi_{j}^{(p)} 
  \big\rangle_{\bL^2_p(\Sigma)}
= \big\langle \sigma^{(p)}_{\ell},\phi_{j,0}^{(p)} 
  \big\rangle_{\bL^2_p(\Sigma)}  
  = c^{(p)}_j a^{(p)}_{j \ell} 
 \int_{\D^*_{2r}} \chi(|z|)\big|z^j\big|^2 |\log(|z|^2)|^p \, 
		  \omega_{\D^*}.
      \end{align}
By \eqref{eq:2.14a} we have
 \begin{align}   \label{eq:2.25a}
   \big\langle \sigma^{(p)}_{\ell},\phi_{j}^{(p)} 
  \big\rangle_{\bL^2_p(\Sigma)} =0 \:
  \text{ for } j\in\{1,\cdots,\delta_p\}, j<\ell.
    \end{align}   
From  \eqref{e:aj} and \eqref{eq:2.25a} we get
   \begin{align}   \label{eq:2.26a}
  a^{(p)}_{j \ell} =0\qquad   \text{ for  }\,  
  j\in\{1,\cdots,\delta_p\},\, \ell\in \{\delta_p+1,\cdots,d_p\} .
    \end{align}  
   By \eqref{eq:2.4a}, \eqref{eq:2.32a}  and \eqref{eq:2.26a},
we get  \eqref{e:sumsigma}.

 Fixing $\ell\in \{1,\ldots,\delta_p\}$, we have on $\D^*_{r}$ 
 by \eqref{eq:2.13a},  \eqref{e:dfphi0}, \eqref{eq:2.23a},    
   \begin{equation}   \label{e:sigmaphi0}
    \big(\sigma^{(p)}_\ell - \phi^{(p)}_{\ell,0}\mathfrak{e}_L^p\big)(z)  
     = \Big(\big( a^{(p)}_{\ell \ell} - c^{(p)}_{\ell}\big)z^{\ell} 
          + \sum_{\substack{j=1\\j\neq \ell}}^{\infty}
a^{(p)}_{j \ell} z^j \Big)\mathfrak{e}_L^p.
   \end{equation}  
  From Lemma \ref{lem_phi0phisigma} and   \eqref{e:aj}
  we have uniformly for $j, \ell\in \{1,\ldots,\delta_p\}$, 
   \begin{equation} \label{eq:2.36a}
   \begin{split}
    a^{(p)}_{j \ell} &= c^{(p)}_{j} \Big((c^{(p)}_{j})^{2}
	\int_{\D^*_{2r}} \chi(|z|)\big|z^j\big|^2 |\log(|z|^2)|^p \, 
		  \omega_{\D^*}\Big)^{-1} 
 \big\langle \sigma^{(p)}_\ell,
   \phi^{(p)}_{j}\big\rangle_{\bL^2_p(\Sigma)}\\
   &= (\delta_{j\ell} +\mathcal{O}(p^{-\infty}))  c^{(p)}_{j}.
   \end{split}
   \end{equation} 
 Thus from \eqref{eq:2.8a}, \eqref{eq:2.36a} we have on $\D^*_{r}$  
   uniformly in $\ell\in \{1,\ldots,\delta_p\}$,
   \begin{multline}   \label{e:headsigma}
     \bigg|\Big( \big( a^{(p)}_{\ell \ell} - c^{(p)}_{\ell}\big)z^{\ell} +
	 \sum_{\substack{j=1\\j\neq \ell}}^{\delta_p}
	 a^{(p)}_{j \ell} z^j\Big)
	\mathfrak{e}_L^p\bigg|_{h^p}^2
= \big|\!\log(|z|^2)\big|^p\bigg|
\sum_{j=1}^{\delta_p} 
\Big( a^{(p)}_{j \ell}- \delta_{j\ell}  c^{(p)}_{j}\Big)
z^j \bigg|^2  \\
  \leq \mathcal{O}(p^{-\infty})\big|\!\log(|z|^2)\big|^p
\delta_p \sum_{j=1}^{\delta_p} 
 (c^{(p)}_{j})^2|z|^{2j}   
 \leq \delta_p\mathcal{O}(p^{-\infty}) B_p^{\D^*}(z),  
   \end{multline}
Now $\delta_p$ can be absorbed in the factor 
$\mathcal{O}(p^{-\infty})$, since $\delta_p=\mathcal{O}(p)$ 
by \eqref{eq:2.13a}.
Combining \eqref{eq:2.32a} with  \eqref{e:headsigma}
we conclude that \eqref{e:sigmaminusphi0} holds
uniformly in $\ell\in \{1,\ldots, \delta_p\}$.
\end{proof}
Since $\delta_p=\mathcal{O}(p)$ 
 and $|\sigma_\ell^{(p)}|_{h^p,z}\leq B_p(z)^{1/2}$, 
 \eqref{e:sigmaminusphi0} also yields
  \begin{equation}   \label{e:dblprod}
\bigg|\sum_{\ell=1}^{\delta_p} 
\big\langle \sigma^{(p)}_{\ell}, 
\phi^{(p)}_{\ell,0} -\sigma^{(p)}_{\ell}\big\rangle_{h^p,z}\bigg|
  = \mathcal{O}(p^{-\infty})\cdot B_p^{\D^*}(z)^{1/2}B_p(z)^{1/2}
  \quad\text{on $\D^*_{cp^{-A}}$}.
  \end{equation}
This way, putting together \eqref{eq:2.16a}, \eqref{e:btail}, 
\eqref{e:sumsigma}, \eqref{e:sigmaminusphi0} and \eqref{e:dblprod}, 
we obtain
  \begin{equation}  \label{eq2.41a}
   \big(1+\mathcal{O}(p^{-\infty})\big)\cdot B_p^{\D^*}(z) 
    = B_p(z)
   + \mathcal{O}(p^{-\infty})\cdot B_p^{\D^*}(z)^{1/2}B_p(z)^{1/2}
   \quad\text{on $\D^*_{cp^{-A}}$},
  \end{equation}
and this implies \eqref{e:apdx2}. 
The proof of Theorem \ref{thm_apdx} is completed. 
 
\subsection{Proof of Lemma \ref{lem_phi0phisigma}}
 \label{eq:s2.2}
 \noindent
  
At first, as $0\leq \chi\leq1$  and ${\rm supp}(\chi) \subset \D^*_{2r}$, 
we get from \eqref{eq:1.6a}, \eqref{eq:2.7a} and \eqref{eq:2.13a},
\begin{multline}\label{eq:3.1a}
\|\phi_{\ell,0}^{(p)}\|_{\bL^2_p(\Sigma)}^{2}
=\|\phi_{\ell,0}^{(p)}\|_{\bL^2_p(\D^{*})}^{2} 
\leq  (c^{(p)}_{\ell})^2 
\int_{\D^*}\chi(|z|) \big|\!\log(|z|^2)\big|^p  |z|^{2\ell}\,\omega_{\D^*}\\
\leq  (c^{(p)}_{\ell})^2 
\int_{\D^*} \big|\!\log(|z|^2)\big|^p  |z|^{2\ell}\,\omega_{\D^*} 
=\|c_\ell^{(p)}z^\ell\|_{\bL^2_p(\D^*)}^{2}=1.
\end{multline}  
This implies the inequalities of the right-hand side  of 
\eqref{e:phi0}. 

We establish now the lower bound of \eqref{e:phi0}.
 For $\ell\in \{1,\ldots,\delta_p\}$ we have
by \eqref{eqn_omegaPcr}, \eqref{eq:2.7a}, 
\eqref{eq:2.13a} and \eqref{e:dfphi0},
   \begin{equation}   \label{e:phi01}
    \begin{aligned}
     1- \|\phi^{(p)}_{\ell,0}\|_{\bL^2_p(\D^*)}^2
        &= \big(c^{(p)}_{\ell}\big)^2 
           \int_{\D^*} \big|\!\log(|z|^2)\big|^p 
  \big\{1-\chi^{2}(|z|)\big\} |z|^{2\ell}\,\omega_{\D^*}    \\
 &= \dfrac{\ell^{p-1}}{ (p-2)!} \int_{r^\beta}^1 
\big|\!\log(t^2)\big|^{p} t^{2\ell} (1-\chi^{2}(t))
\frac{2 t dt}{t^2 \big|\!\log(t^2)\big|^{2}}\\
&\overset{u=-2\ell \log t}{=}
\dfrac{1}{ (p-2)!} \int_0^{2\ell \beta |\log r|} 
u^{p-2} e^{-u} \Big(1-\chi^{2}(e^{-u/(2\ell)})\Big) du\\
&\leq \dfrac{1}{ (p-2)!} \int_0^{2\delta_p \beta |\log r|} 
u^{p-2} e^{-u} du.
  \end{aligned}
   \end{equation}
The function $u \mapsto \log u - u$
is strictly increasing on $(0,1]$ and equals $-1$ at $u=1$, hence
\begin{equation}   \label{eq:3.7a}
\log \beta -\beta <-1.
\end{equation}
As $ u^{p-2} e^{-u}$ is strictly increasing on $[0, p-2]$, 
and $2\delta_p  |\log r| \leq p-2$ (by \eqref{eq:2.13a}),
so \eqref{eq:2.34a} and \eqref{eq:3.7a} imply 
    \begin{equation}   \label{eq:3.6a}
    \begin{split}
\dfrac{1}{ (p-2)!} \int_0^{2\delta_p \beta |\log r|} 
&u^{p-2}  e^{-u} du
\leq \dfrac{1}{ (p-2)!} \int_0^{(p-2) \beta } 
u^{p-2} e^{-u} du\\
&\leq \frac{(p-2)^{p-2}}{(p-2)!}
           e^{(p-2)(\log \beta-\beta)} (p-2)\beta\\
&=  \Big(\dfrac{p-2}{2\pi} \Big)^{1/2} \beta
 \Big(1+ \mathcal{O}(p^{-1})\Big) e^{(p-2)( \log \beta -\beta +1)}\\
&=  \mathcal{O}(p^{-\infty}).	
 \end{split}	
   \end{equation}
Combining \eqref{e:phi01} 
and \eqref{eq:3.6a} we obtain that the first inequality of 
\eqref{e:phi0} holds uniformly in $\ell\in \{1,\ldots,\delta_p\}$.

We move on to \eqref{e:phi0phi} and we first estimate 
  $\|\phi^{(p)}_{\ell} - \phi^{(p)}_{\ell,0}\|_{\bL^2(\D^*_{3r})}$. 
Using the identification \eqref{eq:1.6a}
as in \cite[(6.1)]{bkp} we denote for $x,y\in \D_{4r}^*$,
 \begin{equation}\label{eq:3.10a}\begin{split}
 & {B}_p(x,y)= \big|\!\log(|y|^2)\big|^{p} \beta^{\Sigma}_p(x,y),\\
& {B}^{\D^*}_p(x,y) 
    = \big|\!\log(|y|^2)\big|^{p} \beta^{\D^*}_p(x,y)
    \text{  with } \beta^{\D^*}_p(x,y) = \frac{1}{2\pi(p-2)!} 
    \sum_{\ell=1}^{\infty} \ell^{p-1} x^{\ell}\overline{y}^{\ell}.
       \end{split}\end{equation}
For $\ell\in \{1,\ldots,\delta_p\}$ set 
 \begin{equation}\label{eq:3.11a}\begin{split}
 & I^{(p)}_{1,\ell}(x)=   \int_{y\in\D^*_{2r}} \big|\!\log(|y|^2)\big|^p 
          \big\{\beta_p^{\Sigma}(x,y) - \beta_p^{\D^*}(x,y)\big\}
                   \chi(|y|) y^\ell\, \omega_{\D^*}(y)  ,  \\
  & I^{(p)}_{2,\ell}(x)= \int_{y\in\D^*} \big|\!\log(|y|^2)\big|^p 
  \beta_p^{\D^*}(x,y)\big\{\chi(|y|)-1\big\} y^\ell\, \omega_{\D^*}(y),\\   
   & I^{(p)}_{3,\ell}(x)= \int_{y\in\D^*} \big|\!\log(|y|^2)\big|^p 
  \beta_p^{\D^*}(x,y) y^\ell\, \omega_{\D^*}(y) = x^\ell,
 \end{split}\end{equation}
where the last equality is a consequence of the reproducing property 
of the Bergman kernel ${B}^{\D^*}_p(\LargerCdot,\LargerCdot)$. 
By the construction of $\phi^{(p)}_{\ell}$, \eqref{eq:2.13a}, and the
reproducing property of $B_p(\LargerCdot,\LargerCdot)$
we have for $x\in \D^*_{4r}$, 
   \begin{equation}   \label{e:phi0dcmp}
   \begin{aligned}
    \phi^{(p)}_{\ell}(x) &=(B_p \phi^{(p)}_{\ell,0})(x)
       =  \int_{y\in \Sigma} B_p(x,y) \phi^{(p)}_{\ell,0}(y)
	   \,\omega_{\Sigma}(y)    \\
&= c_\ell^{(p)} \int_{y\in\D^*_{2r}} \big|\!\log(|y|^2)\big|^p 
       \beta_p^{\Sigma}(x,y)\chi(|y|) y^\ell\, \omega_{\D^*}(y)  \\
    &   = c_\ell^{(p)} \Big(  I^{(p)}_{1,\ell}(x)+ I^{(p)}_{2,\ell}(x)
	+ I^{(p)}_{3,\ell}(x)\Big).
   \end{aligned}
   \end{equation}
Now \cite[Theorem 1.1 or (6.23)]{bkp} and \eqref{eq:2.7a} yield for
fixed $\nu>0$ and $m>0$ and for any $x\in \D^*_{4r}$, $p\geq 2$,
   \begin{align} \label{eq:3.14a}\begin{split}
 \Big|I^{(p)}_{1,\ell} (x)\Big| 
 &\leq C(m,\nu) p^{-m}\big|\!\log(|x|^2)\big|^{-\nu-p/2}
\int_{y\in \D^*_{2r}} \big|\!\log(|y|^2)\big|^{-\nu+p/2}
            \chi(|y|)|y|^{\ell}\, \omega_{\D^*}(y)    \\
&\leq C(m,\nu) p^{-m}\big|\!\log(|x|^2)\big|^{-\nu-p/2} \\
     &     \quad
    \cdot\Big(\int_{\D^*}\big|\!\log(|y|^2)\big|^{p}|y|^{2\ell}\,
	\omega_{\D^*}(y)\Big)^{1/2}
           \Big(\int_{\D^*}\big|\!\log(|y|^2)\big|^{-2\nu}\chi^{2}(|y|)\,
		   \omega_{\D^*}(y)\Big)^{1/2} \\
     &  =  C'(m,\nu) p^{-m}\big|\!\log(|x|^2)\big|^{-\nu-p/2}
	 (c^{(p)}_\ell)^{-1}.
 \end{split}  \end{align}
 Keeping $\nu$ fixed and varying $m$ in 
 \eqref{eq:3.14a} we obtain the following uniform estimate in 
 $\ell\in \{1,\ldots,\delta_p\}$,
  \begin{equation}   \label{e:phi0phi2}
 \Big\| c_\ell^{(p)}I^{(p)}_{1,\ell}  \Big\|_{\bL^2_p(\D^*_{3r})} 
     = \mathcal{O}(p^{-\infty}).
  \end{equation} 
By circle symmetry first  and \eqref{eq:2.7a}, \eqref{e:phi0}, 
  \eqref{eq:3.10a} and \eqref{eq:3.11a}   
we obtain,
   \begin{align} \label{eq:3.12a}\begin{split}
  I^{(p)}_{2,\ell}(x)   &= (c_{\ell}^{(p)})^{2}
          \bigg[\int_{y\in\D^*} \big|\!\log(|y|^2)\big|^p
 \big\{\chi(|y|)-1\big\}|y|^{2\ell}\,\omega_{\D^*}(y)\bigg]x^\ell         
      = \mathcal{O}(p^{-\infty})\cdot x^\ell,            
 \end{split}  \end{align}
uniformly in $\ell\in \{1,\ldots,\delta_p\}$. 
Since $\|c^{(p)}_\ell x^\ell\|_{\bL^2_p(\D^*_{3r})}
  \leq \|c^{(p)}_\ell x^\ell\|_{\bL^2_p(\D^*)}=1$, 
  this tells us already that 
   \begin{equation}   \label{e:phi0phi1}
 \Big\| c_\ell^{(p)}I^{(p)}_{2,\ell}  \Big\|_{\bL^2_p(\D^*_{3r})} 
   = \mathcal{O}(p^{-\infty}).
   \end{equation}
   
Since $0\leq 1-\chi\leq 1$ and $1-\chi(t)=0$ for $t\leq r^\beta$,
we get by \eqref{eq:2.7a}, \eqref{eq:3.6a} and \eqref{eq:3.11a}, as in 
\eqref{e:phi01}, that for $\ell\in \{1,\ldots,\delta_p\}$ the following holds,
    \begin{multline}  \label{e:phi0phi3}
  \Big\|c_\ell^{(p)}  I^{(p)}_{3,\ell} (x)
  - \phi^{(p)}_{\ell,0}(x)\Big\|_{\bL^2_p(\D^*_{3r})}^2  \leq 
  \Big\|c_\ell^{(p)}\big(1-\chi(|x|)\big)x^\ell\Big\|_{\bL^2_p(\D^*)}^2\\
= \dfrac{\ell^{p-1}}{ (p-2)!} \int_{r^\beta}^1 
\big|\!\log(t^2)\big|^{p} t^{2\ell} (1-\chi(t))^2 
\frac{2 t dt}{t^2 \big|\!\log(t^2)\big|^{2}}\\
\overset{u=-2\ell \log t}{=} 
\dfrac{1}{ (p-2)!} \int_0^{2\ell \beta |\log r|} 
u^{p-2} e^{-u} \Big(1-\chi(e^{-u/(2\ell)})\Big)^2 du\\
\leq \dfrac{1}{ (p-2)!} \int_0^{2\delta_p \beta |\log r|} 
u^{p-2} e^{-u} du =  \mathcal{O}(p^{-\infty}).
  \end{multline}
 By \eqref{e:phi0dcmp},  \eqref{e:phi0phi2},  \eqref{e:phi0phi1}
  and \eqref{e:phi0phi3}  
   we get the following estimate uniformly in $\ell\in \{1,\ldots,\delta_p\}$,
   \begin{equation}   \label{e:phi0phi4}
    \|\phi^{(p)}_{\ell,0} - \phi^{(p)}_{\ell}\|_{\bL^2_p(\D^*_{3r})} 
	= \mathcal{O}(p^{-\infty}). 
   \end{equation}
A weak form of \cite[Corollary 6.1]{bkp} tells us that for any 
$k\in \N$, $\varepsilon >0$, there exists $C>0$ such that
\begin{align}\label{eq:3.26a}
|B_p(x,y)|\leq C p^{-k}  \quad \text{for }d(x,y)>\varepsilon,\:p\geq 2.
\end{align}
 By \eqref{eq:3.1a},  \eqref{e:phi0dcmp} and  \eqref{eq:3.26a}, 
   \begin{equation}\label{eq:3.27a}
\big\|\phi^{(p)}_{\ell}\big\|^2_{\bL^2_p(\Sigma\smallsetminus\D^*_{3r})} 
\leq C p^{-2k} \int_{\Sigma\setminus \D^{*}_{3r}}\omega_{\Sigma}
\int_{\D^*_{2r}}| \phi_{\ell,0}^p (y)|_{h^p}^2 \omega_{\Sigma}(y)
\leq  C p^{-2k} \int_{\Sigma}\omega_{\Sigma}.  
   \end{equation}
From \eqref{e:phi0phi4} and \eqref{eq:3.27a} we have
uniformly in $\ell\in\{1,\ldots,\delta_p\}$, 
   \begin{equation}\label{eq:3.28a}
    \big\|\phi^{(p)}_{\ell}-\phi^{(p)}_{\ell,0}\big\|_{\bL^2_p(\Sigma)}^2
 =     \|\phi^{(p)}_{\ell}-\phi^{(p)}_{\ell,0}\|^2_{\bL^2_p(\D^*_{3r})}
 + \big\|\phi^{(p)}_{\ell}
 \big\|^2_{\bL^2_p(\Sigma\smallsetminus\D^*_{3r})} 
	= \mathcal{O}(p^{-\infty}).
   \end{equation}
 By \eqref{e:phi0}  
 and \eqref{eq:3.28a}, 
 as $\phi^{(p)}_{j}- \phi^{(p)}_{j,0}$  is orthogonal to 
 $H^{0}_{(2)}(\Sigma, L^{p})$, we have
uniformly in $j,\ell\in \{1,\ldots,\delta_p\}$, 
   \begin{align}\label{eq:3.29a}\begin{split}
\big\langle \phi^{(p)}_{j}, \phi^{(p)}_{\ell}\big\rangle_{\bL^2_p(\Sigma)}
&=\big\langle \phi^{(p)}_{j,0}, \phi^{(p)}_{\ell}
	\big\rangle_{\bL^2_p(\Sigma)}  \\
        &= \big\langle \phi^{(p)}_{j,0}, \phi^{(p)}_{\ell,0}
		\big\rangle_{\bL^2_p(\D^*_{2r})} 
+ \big\langle \phi^{(p)}_{j,0}, \phi^{(p)}_{\ell}- \phi^{(p)}_{\ell,0}
\big\rangle_{\bL^2_p(\Sigma)} \\
 &= \delta_{j\ell}   + \mathcal{O}(p^{-\infty}).
\end{split}   \end{align}
Note that the circle symmetry and \eqref{e:dfphi0} imply that
 $\big\langle \phi^{(p)}_{j,0}, \phi^{(p)}_{\ell,0}
 \big\rangle_{\bL^2_p(\D^*_{2r})}=0$ if $j\neq \ell$.
We now observe that the Gram-Schmidt orthonormalization
 $(\sigma_{\ell}^{(p)})_{1\leq \ell \leq \delta_p}$ of 
  the ``almost-orthonormal'' family 
  $(\phi_{\ell}^{(p)})_{1\leq \ell \leq \delta_p}$
  is the normalization of 
   \begin{equation}\label{eq:3.32a}
{\sigma'}_{\ell}^{(p)}= 	    \phi^{(p)}_{\ell}
- \sum_{k=1}^{\ell-1} \frac{\big\langle  \phi^{(p)}_{\ell}, 
\phi^{(p)}_{k}\big\rangle_{\bL^2_p(\Sigma)}}{
	   \big\langle  \phi^{(p)}_{k}, 
\phi^{(p)}_{k}\big\rangle_{\bL^2_p(\Sigma)}} \phi^{(p)}_{k}.
    \end{equation}   
Now \eqref{e:alphabeta}, \eqref{eq:3.28a},
\eqref{eq:3.29a} and \eqref{eq:3.32a}  
yield \eqref{e:phi0phi}.  
This completes the proof of Lemma \ref{lem_phi0phisigma}.

 \comment{~
 Taking the \textit{smoothness of the quotient} 
 $\frac{B_p}{B_p^{\D^*}}$ \textit{over} $0\in \D$ 
 into account, it is now tempting to state: 
  \begin{equation}   \label{e:diffquot}
   \text{for all $m\geq 1$ 
 and for $D_1,\ldots,D_m\in\Big\{\frac{\partial\,}{\partial z}, 
 \frac{\partial\,}{\partial \overline{z}}\Big\},$
         }
    \sup_{\overline{V_1}\cup\cdots\cup\overline{V_N}}
     \Big|(D_1\cdots D_m)\frac{B_p}{B_p^{\D^*}}\Big| 
	 = \mathcal{O}(p^{-\infty}).
  \end{equation}

 \begin{thm}   \label{thm:diffquot}
  Estimates \eqref{e:diffquot} hold. 
 \end{thm}
 This answers a question raised by J.-M. Bismut in CIRM in October 2018.
 } 
 \section{$C^{k}$-estimate of the quotient of Bergman kernels}
 \label{eq:s3}
 
The proof of Theorem \ref{thm:diffquot} follows the same strategy 
  as in Section \ref{eq:s2}
  (use of the orthonormal basis $(\sigma_j^{(p)})_{1\leq j\leq d_p}$), 
  but with some play on the parameters (in particular, 
   the truncation floor $\delta_p$ of Step 1.\ in the outline of the proof 
  of Theorem \ref{thm_apdx}).   
  Some precisions on this basis are also needed:  
  we'll see more precisely that in some sense, 
  and provided relevant choices along the construction, 
  the head terms $\sigma_{\ell}^{(p)}$, $1\leq j\leq \delta_p$ 
  are much closer to their counterparts $c_{\ell}^{(p)}z^{\ell}$ of $\D^*$ 
  than sketched above. 
 
  This section is organized as follows. In Section \ref{eq:s3.1},
 we establish a  refinement of the integral estimate
 Lemma \ref{lem_phi0phisigma} which is again deduced
 from \cite{bkp}.
 In Section \ref{eq:s3.2},
 we establish Theorem \ref{thm:diffquot} 
 by using Lemma \ref{prop:rfnd_estmt}.
  
 \subsection{A refined integral estimate}
 \label{eq:s3.1}

  To establish  Theorem \ref{thm:diffquot}, we follow 
  Steps 1.\ to 4.\ in the outline of the proof of Theorem \ref{thm_apdx} 
  by modifying $\delta_{p}$, thus refining Lemma \ref{lem_phi0phisigma}
  to Lemma \ref{prop:rfnd_estmt} below.
  

Let $\kappa>0$ fixed. We start by choosing $c(\kappa) \in (0,e^{-1})$ 
so that 
     \begin{equation}   \label{e:ckappa}
      \log (c(\kappa)) 
	  \leq -1-2\kappa.
     \end{equation}
Then we replace $\delta_{p}$ in \eqref{eq:2.13a} by 
     \begin{equation}   \label{eq:3.35a}   
	\delta_p'=\delta'_p(\kappa) 
= \Big\lfloor\frac{(p-2)c(\kappa)}{2|\!\log r|}\Big\rfloor-2. 
    \end{equation}
   
   \begin{lem}   \label{prop:rfnd_estmt}
There exists $C = C(\kappa)>0$ such that 
    for all $p\gg 1$ and $\ell \in \{1,\ldots,\delta'_p\}$, 
     \begin{equation}   \label{e:rfnd_estmt}
      \left\|\sigma_{\ell}^{(p)}
 - c_{\ell}^{(p)} \chi(|z|)z^{\ell}\mathfrak{e}^p_L 
 \right\|_{\bL^2_p(\Sigma)} 
         \leq Cp\,    e^{-\kappa p}. 
     \end{equation}
    Moreover, $(\sigma_{\ell}^{(p)})_{1\leq j\leq d_p}$ 
    is in echelon form up to rank $\delta'_p$, 
    in the sense that 
    if $\ell=1,\ldots,\delta'_p$, 
    then $\sigma_{\ell}^{(p)}$ admits an expansion 
     \begin{equation}   \label{e:echelon1}
      \sigma_{\ell}^{(p)} 
         = \Big(\sum_{q=\ell}^{\infty} 
                   a^{(p)}_{q\ell} z^q
           \Big)\mathfrak{e}^p_L \quad \text{ on } \D_{4r}^{*},
     \end{equation}
    and if $\ell= \delta_p'+1,\ldots,d_p$, 
    then $\sigma_{\ell}^{(p)}$ admits an expansion 
     \begin{equation}   \label{e:echelon2}
      \sigma_{\ell}^{(p)} 
         = \Big(\sum_{q=\delta'_p+1}^{\infty} 
                   a^{(p)}_{q\ell} z^q 
           \Big)\mathfrak{e}^p_L \quad \text{ on } \D_{4r}^{*}.
     \end{equation}
   \end{lem}
  As will be seen, 
  estimate \eqref{e:rfnd_estmt} is 
  directly related to the play on $\delta_p'$, 
  whereas the echelon property as such is not,
  and \eqref{e:echelon1}, \eqref{e:echelon2} are a direct consequence of 
\eqref{e:aj} and \eqref{eq:2.25a}.   
  Moreover, no estimate is given on the $\sigma_{\ell}^{(p)}$ 
  for $\ell\geq \delta'_p+1$ in the above statement; 
  as in the proof of Theorem \ref{thm_apdx}, 
  it turns out that we content ourselves with rather rough estimates 
  on these tail sections. 
 	 
 
\begin{proof}  [Proof of Lemma \ref{prop:rfnd_estmt}]
 Let $\overline\partial^{L^{p}*}$ be the formal adjoint of 
$\overline\partial^{L^{p}}$ on $(C^{\infty}_{0}(\Sigma, L^{p}), 
\|\quad\|_{\bL^2_p(\Sigma)})$. Then 
 $\Box_{p}=\overline\partial^{L^{p}*}\overline\partial^{L^{p}}
 : C^{\infty}_{0}(\Sigma, L^{p})\to C^{\infty}_{0}(\Sigma, L^{p})$
 is the Kodaira Laplacian  on $L^{p}$ and 
 \begin{align}\label{eq:3.68a}
	 \ker\square_p= H_{(2)}^{0}(\Sigma, L^{p}).
	 \end{align}

Observe that the construction of the $\phi_{\ell}^{(p)}$, 
  $\ell = 1,\ldots,\delta'_p$, 
  following Steps 1. to 4. of the proof of Theorem \ref{thm_apdx}  
  can be led alternatively by the following principle:
   \begin{itemize}
    \item[1'.]   with the cut-off function $\chi$ in \eqref{eq:2.13a},
	for $\ell = 1,\ldots,\delta'_p$,  set 
  \begin{align}\label{eq:3.70a}
	  \phi_{0,\ell}^{(p)}:=\phi_{\ell,0}^{(p)} 
	  = c^{(p)}_{\ell}\chi(|z|)z^{\ell}\mathfrak{e}_L^p .
\end{align}	  
    \item[2'.] give an explicit estimate of 
 $\big\|\square_p\phi_{0,\ell}^{(p)}\big\|_{\bL^2_p(\Sigma)}$.		
    \item[3'.]  we correct $\phi_{0,\ell}^{(p)}$ into holomorphic 
	$\bL^{2}$-section $\phi_{\ell}^{(p)}$ of $L^{p}$, by orthogonal 
	$\bL^2_p(\Sigma)$-projection.
we use the spectral gap property \cite[Cor.5.2]{bkp} 
(as a direct consequence of \cite[Theorem 6.1.1]{mm})
     together with the step 2'     to get \eqref{e:rfnd_estmt}.
   \end{itemize}

 \noindent
 \textit{Step 1'. } 
 We compute, by \eqref{eq:2.7a} and \eqref{eq:2.13a}, 
 as in \eqref{e:phi01}, for $\ell=1,\ldots,\delta'_p$, 
  \begin{equation} \label{eq:3.71a}
    \begin{aligned}
0\leq   1 - \big\|c_{\ell}^{(p)}z^{\ell} \chi(|z|)\mathfrak{e}_L^p
   \big\|_{\bL^2_p(\Sigma)}^2 
 &=\dfrac{1}{ (p-2)!} \int_0^{2\ell \beta |\log r|} 
u^{p-2} e^{-u} \Big(1-\chi^{2}(e^{-u/(2\ell)})\Big) du\\
&\leq \dfrac{1}{ (p-2)!} \int_0^{2\delta'_p \beta |\log r|} 
u^{p-2} e^{-u} du.
   \end{aligned}
   \end{equation}
As $ u^{p-2} e^{-u}$ is strictly increasing on $[0, p-2]$
and $\log \beta<0$, and from \eqref{eq:3.35a},
$2(\delta'_p+2)  |\log r| \leq (p-2)c(\kappa)$,
by \eqref{eq:2.34a}, 
\eqref{e:ckappa}, 
we get a refinement of \eqref{eq:3.6a}, 
    \begin{equation}   \label{eq:3.73a}
    \begin{split}
\dfrac{1}{ (p-2)!} \int_0^{2\delta'_p \beta |\log r|} 
 &    u^{p-2} e^{-u} du
\leq \dfrac{1}{ (p-2)!} \int_0^{(p-2) c(\kappa)\beta } 
u^{p-2} e^{-u} du\\
&\leq \frac{(p-2)^{p-2}}{(p-2)!}
  e^{(p-2)(\log (c(\kappa)\beta)-c(\kappa)\beta )} (p-2)c(\kappa) \beta\\
&=  \Big(\dfrac{p-2}{2\pi} \Big)^{1/2} c(\kappa)\beta
 \Big(1+ \mathcal{O}(p^{-1})\Big) 
 e^{(p-2)( \log (c(\kappa)\beta) -c(\kappa)\beta +1)}\\
& =  \mathcal{O}(e^{-2\kappa \, p}).		
 \end{split}
   \end{equation}
  From \eqref{eq:3.71a} and \eqref{eq:3.73a},  
 uniformly in $\ell\in \{1,\ldots,\delta'_p\}$,
  \begin{equation}   \label{e:|phi0|}
 \big\|  \phi_{0,\ell}^{(p)}
   \big\|_{\bL^2_p(\Sigma)}^{2} 
   = \big\|c_{\ell}^{(p)}z^{\ell} \chi(|z|)\mathfrak{e}_L^p
   \big\|_{\bL^2_p(\Sigma)}^{2} 
    = 1 + \mathcal{O}(	e^{-2\kappa p}).
  \end{equation}
 \textit{Step 2'. }
 Recall from \cite[(4.14), (4.15) or (4.30)]{bkp} 
 that on $\D_{2r}^*$ (seen in $\Sigma$), 
  \begin{equation}\label{eq:3.74a}
   \square_p(\LargerCdot\,\mathfrak{e}_L^p) 
= \Big(- |z|^2\log^2(|z|^2)\frac{\partial^2\LargerCdot}
{\partial z\partial\bar{z}} 
- p\bar{z}\log(|z|^2)\frac{\partial\LargerCdot}{\partial\bar{z}}\Big)
\mathfrak{e}_L^p.
  \end{equation}
Hence we obtain from \eqref{eq:3.70a} and \eqref{eq:3.74a}, 
for $\ell=1,\ldots,\delta'_p$, 
  \begin{equation}\label{eq:3.75a}
   \square_p\phi_{0,\ell}^{(p)} 
      = c^{(p)}_{\ell}
        \Big(- |z|^2\log^2(|z|^2)\frac{\partial^2}{\partial z\partial\bar{z}}
             \big(\chi(|z|)z^{\ell}\big)
         - p\, \bar{z}\log(|z|^2)\frac{\partial}{\partial\bar{z}}
             \big(\chi(|z|)z^{\ell}\big)
         \Big)\mathfrak{e}_L^p.
  \end{equation}
Since $\frac{\partial}{\partial\bar{z}}[\chi(|z|)z^{\ell}]
      = \big(\frac{\partial}{\partial\bar{z}}\chi(|z|)\big)z^{\ell}
      = \frac{|z|}{2\bar{z}}\chi'(|z|)z^{\ell}$,
we have
 \begin{equation}\label{eq:3.76a}
   \frac{\partial^2}{\partial 
   z\partial\bar{z}}\Big[\chi(|z|)z^{\ell}\Big]  
    = \frac{2\ell+1}{4|z|}z^{\ell}\chi'(|z|)
        + \frac14 z^{\ell}\chi''(|z|),
		 \end{equation}
which yields
  \begin{multline}\label{eq:3.77a}
   \square_p\phi_{0,\ell}^{(p)} = c^{(p)}_{\ell}
        \Big(- \frac{2\ell+1}{4}|z|z^{\ell}\log^2(|z|^2) \chi'(|z|)      \\
 - \frac14|z|^2z^{\ell}\log^2(|z|^2) \chi''(|z|)
-\frac{p}{2}|z|z^{\ell}\log(|z|^2) \chi'(|z|)\Big)\mathfrak{e}_L^p
  \end{multline}
 on $\D^*_{2r}$, and this readily extends to the whole $\Sigma$. 
 Therefore, 
  \begin{equation}   \label{e:Kodphi0}
  \begin{aligned}
   \big\|\square_p\phi_{0,\ell}^{(p)}&\big\|_{\bL^2_p(\Sigma)} 
 \leq  c^{(p)}_{\ell}
\Big(\frac{2\ell+1}{4}\big\||z|^{\ell+1}\log^2(|z|^2) \chi'(|z|)
\big\|_{\bL^2_p(\D^*)}     \\
&+ \frac14\big\||z|^{\ell+2}\log^2(|z|^2) \chi''(|z|)\big\|_{\bL^2_p(\D^*)}       
+ \frac{p}{2} \big\||z|^{\ell+1}\log(|z|^2) \chi'(|z|)\big\|_{\bL^2_p(\D^*)} 
           \Big).
  \end{aligned}
  \end{equation}
 Using nonetheless arguments similar to those of Step 1'. 
 above, 
 we claim that there exists  $C>0$ such that 
 for all $p\gg1$ and  $\ell=1,\ldots,\delta'_p$, 
  \begin{equation}   \label{e:|Kodphi0|}
  \begin{aligned}
   \big\||z|^{\ell+1}\log^2(|z|^2) \chi'(|z|)\big\|_{\bL^2_p(\D^*)} 
     &\leq C(c^{(p+4)}_{\ell+1})^{-1}  e^{-\kappa p},  \\
   \big\||z|^{\ell+2}\log^2(|z|^2) \chi''(|z|)\big\|_{\bL^2_p(\D^*)} 
     &\leq C(c^{(p+4)}_{\ell+2})^{-1}  e^{-\kappa p},  \\
  \big\||z|^{\ell+1}\log(|z|^2) \chi'(|z|)\big\|_{\bL^2_p(\D^*)} 
    &\leq C(c^{(p+2)}_{\ell+1})^{-1}  e^{-\kappa p}.                                     
  \end{aligned}
  \end{equation}
 Indeed, 
 from \eqref{eq:3.35a} and since
$2(\delta'_p+2)  |\log r| \leq (p-2)c(\kappa)$,
applying \eqref{eq:3.73a} for $p+4$
 as in \eqref{e:phi01}, we get with $C_{0}=\sup_{[0,1]}|\chi'|$, 
  \begin{align}\label{eq:3.80a}\begin{split}
   \big\| c^{(p+4)}_{\ell+1}|z|^{\ell+1}& \log^2(|z|^2) 
   \chi'(|z|)\big\|_{\bL^2_p(\D^*)}^2 \\
&= (c^{(p+4)}_{\ell+1})^{2}   
    \int_{r^\beta}^1 
\big|\!\log(t^2)\big|^{p+4} t^{2\ell+2} \chi'(t)^{2}
\frac{4\pi t dt}{t^2 \big|\!\log(t^2)\big|^{2}}\\
 &\overset{u=-2(\ell +1)\log t}{=\joinrel=} 
\dfrac{1}{ (p+2)!} \int_0^{2(\ell+1) \beta |\log r|} 
u^{p+2} e^{-u} \Big(\chi'\Big(e^{-\frac{u}{2(\ell+1)}}\Big)\Big)^{2}du\\
&\leq \dfrac{C_{0}^{2}}{ (p+2)!} \int_0^{2(\delta'_p+1) \beta |\log r|} 
u^{p+2} e^{-u} du\\
&\leq \dfrac{C_{0}^{2}}{ (p+2)!} \int_0^{(p+2) c(\kappa)\beta } 
u^{p+2} e^{-u} du=  \mathcal{O}(e^{-2\kappa \, p}).
\end{split}  \end{align}
Consequently, by 
 \eqref{e:Kodphi0} and \eqref{e:|Kodphi0|}, there exists $C>0$ such that
  for all $p\gg1$ and $\ell=1,\ldots,\delta'_p$, 
  \begin{equation}   \label{e:|Kophi0|2}
  \begin{split}
   &\big\|\square_{p}\phi_{0,\ell}^{(p)}\big\|_{\bL^{2}_{p}}
      \leq Cc^{(p)}_{\ell}\Big(\ell(c^{(p+4)}_{\ell+1})^{-1}
                               +(c^{(p+4)}_{\ell+2})^{-1}
                               +p(c^{(p+2)}_{\ell+1})^{-1}\Big)e^{-\kappa p}\\
&= C\left(\frac{\ell^{p-1}}{(p-2)!}\right)^{\!\frac{1}{2}}	
\left(\ell \left(\frac{(p+2)!}{(\ell+1)^{p+3}}\right)^{\!\frac{1}{2}}	
+ \left(\frac{(p+2)!}{(\ell+2)^{p+3}}\right)^{\!\frac{1}{2}}	
+ p \left(\frac{p!}{(\ell+1)^{p+1}}\right)^{\!\frac{1}{2}}\right) 
e^{-\kappa p}\\
&\leq C p^{2} e^{-\kappa p}.
\end{split}
  \end{equation}
  \noindent
 \textit{Step 3'. } 
 Recall that the spectral gap property \cite[Corollary 5.2]{bkp} 
tells us that there exists $C_{1}>0$ such that for all $p\gg 1$ we have
\begin{align}\label{eq:3.85a}
\text{ Spec}(\square_{p})\subset \{0\} \cup [C_{1} p, +\infty).
  \end{align}	
For $\ell\in\{1,\ldots,\delta'_p\}$ 
let $\psi_{0,\ell}^{(p)} \in \bL^{2}_p(\Sigma)$ such that 
$\psi_{0,\ell}^{(p)} \perp H^{0}_{(2)}(\Sigma, L^{p})$ and 
$\square_{p}\psi_{0,\ell}^{(p)} 
=\square_{p}\phi_{0,\ell}^{(p)}$. 
Then by \eqref{eq:3.68a},
\begin{align}\label{eq:3.84a}
\phi_{\ell}^{(p)}=\phi_{0,\ell}^{(p)}-\psi_{0,\ell}^{(p)}.
 \end{align}	  
By \eqref{e:|Kophi0|2}, \eqref{eq:3.85a} and \eqref{eq:3.84a} we get
\begin{equation}   \label{e:|psi|}
\big\|\phi_{\ell}^{(p)} - \phi_{0,\ell}^{(p)}\|_{\bL^{2}_{p}(\Sigma)} 
 =    \big\|\psi_{0,\ell}^{(p)} \|_{\bL^{2}_{p}(\Sigma)} 
\leq (C_{1} p)^{-1}  \|\square_{p}\phi_{0,\ell}^{(p)}
\|_{\bL^{2}_{p}(\Sigma)} 
\leq C C_{1}^{-1} p\, e^{-\kappa p},  
   \end{equation}
 uniformly in $\ell=1,\ldots,\delta'_p$.   
   Note that  \eqref{e:|phi0|} can be reformulated as 
  \begin{equation}\label{eq:3.87a}
   \big\langle \phi_{0,\ell}^{(p)}, \phi_{0,j}^{(p)}
   \big\rangle_{\bL^{2}_{p}(\Sigma)} 
      = \delta_{j\ell}\big(1+\mathcal{O}(e^{-2\kappa p})\big),
  \end{equation}
(the case $\ell\neq j$ provides 0 by circle symmetry).
Thus \eqref{eq:3.84a}, \eqref{e:|psi|} and \eqref{eq:3.87a} entail
  \begin{equation}\label{eq:3.88a}
   \big\langle \phi_{\ell}^{(p)}, \phi_{j}^{(p)}
   \big\rangle_{\bL^{2}_{p}(\Sigma)}
      = \big\langle \phi_{0,\ell}^{(p)}, \phi_{0,j}^{(p)}
	  \big\rangle_{\bL^{2}_{p}(\Sigma)}
- \big\langle \psi_{0,\ell}^{(p)}, \psi_{0,j}^{(p)}
\big\rangle_{\bL^{2}_{p}(\Sigma)}
      = \delta_{j\ell} + \mathcal{O}(p^{2} e^{-2\kappa p}),
  \end{equation}
 uniformly in $\ell, j=1,\ldots,\delta'_p$.
Because $(\sigma_{\ell}^{(p)})_{1\leq \ell\leq\delta'_p}$ 
is obtained by the Gram-Schmidt orthonormalisation of 
$(\phi_{\ell}^{(p)})_{1\leq \ell\leq \delta'_p}$ 
(which is a $\delta'_p=\mathcal{O}(p)$ process) 
we infer from \eqref{eq:3.32a} and \eqref{eq:3.88a} that
  \begin{equation}\label{eq:3.92a}
    \big\|\sigma_{\ell}'^{(p)} 
          - \phi_{\ell}^{(p)}\|_{\bL^{2}_p(\Sigma)} 
  =\mathcal{O}(p^{3} e^{-2\kappa p}),\qquad 
   \big\|\sigma_{\ell}'^{(p)} \|_{\bL^{2}_p(\Sigma)} 
	= 1+ \mathcal{O}(p^{3} e^{-2\kappa p}).
\end{equation}
Since $\sigma_{\ell}^{(p)}= \sigma_{\ell}'^{(p)} / 
\big\|\sigma_{\ell}'^{(p)} \|_{\bL^{2}_p(\Sigma)}$, 
we conclude from \eqref{eq:3.92a} that there exists $C>0$ such
that for $p\gg1$, 
  \begin{equation}\label{eq:3.89a}
    \big\|\sigma_{\ell}^{(p)} 
          - \phi_{\ell}^{(p)}\|_{\bL^{2}_p(\Sigma)} 
\leq \Big |  \big\|\sigma_{\ell}'^{(p)} \|_{\bL^{2}_p(\Sigma)} -1\Big|
+  \big\|\sigma_{\ell}'^{(p)} 
          - \phi_{\ell}^{(p)}\|_{\bL^{2}_p(\Sigma)} 
		  \leq Cp^{3}   e^{-2 \kappa p}, 
  \end{equation}
hence, by \eqref{e:|psi|} and \eqref{eq:3.89a}
that we have uniformly in $\ell=1,\ldots,\delta'_p$ for $p\gg1$, 
  \begin{equation}\label{eq:3.90a}
\big\|\sigma_{\ell}^{(p)} 
- c_{\ell}^{(p)}\chi(|z|)z^{\ell}\mathfrak{e}_L^p 
   \big\|_{\bL^{2}_p(\Sigma)} 
= \big\|\sigma_{\ell}^{(p)} - \phi_{0,\ell}^{(p)}\|_{\bL^{2}_p(\Sigma)} 
\leq Cp \, e^{-\kappa p}.
   \end{equation}

 \textit{Echelon property. --- } 
We use the expansion \eqref{eq:2.23a} of $\sigma^p_{\ell}$ 
on $\D^*_{4r}$. 
 By construction,
 $\phi_{j}^{(p)}\in {\rm Span}\{\sigma_1^{(p)},\ldots,\sigma_j^{(p)}\}$
 for $1\leq j \leq \delta'_{p}$, so if $j<\ell$, then 
 $\phi_{j}^{(p)}\perp_{\bL^2_p(\Sigma)} \sigma^{(p)}_{\ell}$, 
  as $\phi_j^{(p)}$ is the $\bL^2_p(\Sigma)$-projection
  of $\phi_{0,j}^{(p)}$ on holomorphic sections.
  Hence we have as in \eqref{eq:2.25a} that
   \begin{equation}   \label{e:<sigma,phi>}
    \big\langle \sigma_{\ell}^{(p)}, \phi_{0,j}^{(p)} 
	\big\rangle_{\bL^2_p(\Sigma)} 
      = \big\langle \sigma_{\ell}^{(p)}, \phi_{j}^{(p)}
	  \big\rangle_{\bL^2_p(\Sigma)} =0\,, \quad \text{  if } j<\ell .
   \end{equation}
Now \eqref{e:aj} and \eqref{e:<sigma,phi>} entail
  \begin{equation}   \label{eq:3.98a}
  a_{j\ell}^{(p)}=0\qquad \text{  if } j<\ell, \,  j\in \{1,\ldots,\delta_p'\}, 
\,   \ell\in \{1,\ldots,d_p\}.
  \end{equation}
From \eqref{eq:2.23a} and \eqref{eq:3.98a} we get 
\eqref{e:echelon1} and \eqref{e:echelon2}.
The proof of Lemma \ref{prop:rfnd_estmt} is completed.
 \end{proof}

The following consequence of Lemma  \ref{prop:rfnd_estmt}
that refines \eqref{eq:2.36a} is very useful in our computations.
 \begin{lem}\label{eq:t3.4}
We have uniformly  for $j, \ell\in \{1,\ldots,\delta'_p\}$, 
  \begin{equation}   \label{e:a_alphaell} \begin{split}
  a_{j\ell}^{(p)} =
   \left \{ \begin{array}{l}
   0 \hspace{3cm} \qquad  \text{ for }   j<\ell; \\
  c_{j}^{(p)} \,\big(\delta_{j\ell}+\mathcal{O}(pe^{-\kappa p})\big) 
   \quad \text{ for } j\geq \ell .
 \end{array} \right.   \end{split}  
   \end{equation}
\end{lem}
\begin{proof}   
First note that by \eqref{eq:3.70a} we have 
 \begin{equation}   \label{e:<sigma,phi0>2}
  \begin{aligned}
 \big\langle \sigma_\ell^{(p)} , \phi_{j}^{(p)}
 \big\rangle_{\bL^2_p(\Sigma)}
= \big\langle \sigma_\ell^{(p)}, \phi_{0,j }^{(p)}
    \big\rangle_{\bL^2_p(\Sigma)}     
= &\big\langle \sigma_\ell^{(p)} -   \phi_{0,\ell}^{(p)} ,
 \phi_{0,j }^{(p)} \big\rangle_{\bL^2_p(\Sigma)} 
 +  \big\langle \phi_{0,\ell}^{(p)},
 \phi_{0,j}^{(p)}\big\rangle_{\bL^2_p(\Sigma)} .
   \end{aligned}
  \end{equation}
Further, \eqref{e:rfnd_estmt}, \eqref{eq:3.87a} and
\eqref{e:<sigma,phi0>2} imply
  \begin{equation}\label{eq:3.55a}
   \big\langle \sigma_\ell^{(p)} , \phi_{j}^{(p)}
\big\rangle_{\bL^2_p(\Sigma)} 
= \mathcal{O}(pe^{-\kappa p}) 
+ \delta_{j \ell}\, \big(1+\mathcal{O}(e^{-2\kappa p})\big).
  \end{equation}
By \eqref{eq:3.1a} and \eqref{e:|phi0|} we have uniformly on 
$j\in \{1,\ldots,\delta'_p\}$,
   \begin{equation}\label{eq:3.53a}
   (c_{j }^{(p)})^2
   \int_{\D^*_{2r}}\big|\!\log(|z|^2)\big|^p |z|^{2j } \chi(|z|)\,
   \omega_{\D^*} 
      = 1+ \mathcal{O}(e^{-2\kappa p}).
  \end{equation}
The first equality of \eqref{eq:2.36a},
\eqref{eq:3.98a}, \eqref{eq:3.55a}
and \eqref{eq:3.53a} entail \eqref{e:a_alphaell}. 
 \end{proof} 
 
 \subsection{Proof of Theorem \ref{thm:diffquot} }
 \label{eq:s3.2}

  We show now how to establish Theorem \ref{thm:diffquot} 
 by using Lemma \ref{prop:rfnd_estmt}. 
  It can be noticed here that while estimate \eqref{e:rfnd_estmt} 
  is essential in establishing Theorem \ref{thm:diffquot}, 
  the echelon property is not, 
  but helps nonetheless clarify some of the upcoming computations. 
  
 The proof goes as follows: 
 we start by explicit computations,
then  use Lemma  \ref{prop:rfnd_estmt} to lead a precise analysis 
of head terms, i.e., all its indices $\leq \delta'_{p}$;
 and recall some rough estimates of tail terms,
 i.e., some of its indices $\geq \delta'_{p}+1$.  
 On some shrinking disc family $\{|z|\leq c'p^{-A'}\}$, 
 we will conclude from \eqref{eq:2.27a} that the tails terms
 can be controlled by $2^{-\alpha' p} B^{\D^*}_{p}$,
 hence on some fixed trivialization disc, as for Theorem \ref{thm_apdx}.

 From  \eqref{eq:3.10a}, for $z\in \D^{*}_{4r}$,  set 
   \begin{equation}\label{eq:4.1a}
   \beta^\Sigma_{p}(z)= \beta^\Sigma_{p}(z,z), \qquad 
    \beta^{\D^*}_{p}(z)= \beta^{\D^*}_{p}(z,z).
   \end{equation}
    By \eqref{eq:3.10a} and \eqref{eq:4.1a}, we have
 \begin{align}\label{eq:4.3a}
   \frac{B_p(z)}{B_p^{\D^*}(z)} 
   =  \frac{\beta_p^{\Sigma}(z)}{\beta_p^{\D^*}(z)} 
= 1 + \big(\beta_p^{\Sigma} - \beta_p^{\D^*}\big)(z)
( \beta_p^{\D^*}(z))^{-1}. 
 \end{align}
 With the notations of Lemma 
 \ref{prop:rfnd_estmt} we compute explicitly on $\D_{4r}^*$.
 By \eqref{e:BFS1}, \eqref{eq:2.8a}, \eqref{eq:2.23a},
\eqref{eq:3.10a} and \eqref{eq:4.1a}, we have     
 \begin{align}\label{eq:4.2a}\begin{split}
 \beta_p^{\Sigma}(z)  
=\sum_{q,s=1}^{\infty}  \bigg(\sum_{\ell=1}^{d_p} 
a_{q\ell}^{(p)}\overline{a_{s\ell}^{(p)}}\bigg) z^q \bar{z}^s ,
\qquad \beta_p^{\D^*}(z) = 
 \sum_{q=1}^{\infty} (c_q^{(p)})^2 |z|^{2q}.   
 \end{split} \end{align}
For $q,s\in \N^{*}$, set
\begin{align}\label{eq:4.7a} 
\epsilon_{qs}=  \sum_{\ell=1}^{d_p} \tfrac{a_{q\ell}^{(p)}}{c_q^{(p)}}
                   \tfrac{\overline{a_{s\ell}^{(p)}}}{c_s^{(p)}} 
                  - \delta_{qs}. 
  \end{align}
From \eqref{eq:4.2a} and \eqref{eq:4.7a},  we get
  \begin{align}\label{eq:4.5a}\begin{split}
   \frac{d}{dz}& \big(
  \beta_p^{\Sigma} - \beta_p^{\D^*} \big)(z) 
      \cdot  \beta_p^{\D^*}(z)  \\
  &=\sum_{q,s=1}^{\infty}
q \bigg(\sum_{\ell=1}^{d_p} a_{q\ell}^{(p)}
	\overline{a_{s\ell}^{(p)}} z^{q-1} \bar{z}^s
 -  \delta_{qs}(c_q^{(p)})^2 z^{q-1}
			  \bar{z}^s\bigg)
         \cdot\sum_{m=1}^{\infty} (c_m^{(p)})^2 |z|^{2m}                                         \\
      &= \sum_{q,s,m=1}^{\infty}  
	  q(c_m^{(p)})^2 c_q^{(p)}c_s^{(p)}  \epsilon_{qs} 
   z^{q+m-1} \bar{z}^{s+m}, 
\end{split}  \end{align}
 and similarly, 
\begin{align}\label{eq:4.6a}\begin{split}
\big(
  \beta_p^{\Sigma} - \beta_p^{\D^*} \big)(z)
  \cdot \frac{d}{dz}\beta_p^{\D^*}(z)          
       = \sum_{q,s,m=1}^{\infty}  
	   m(c_m^{(p)})^2 c_q^{(p)}c_s^{(p)}\epsilon_{qs} 
             z^{q+m-1} \bar{z}^{s+m}.
\end{split}  \end{align}
 From  \eqref{eq:4.3a}, \eqref{eq:4.5a} and
 \eqref{eq:4.6a}, we get
  \begin{equation}\label{eq:4.8a}
 \frac{d}{dz}\frac{B_p}{B_p^{\D^*}} (z)
 =  (\beta_p^{\D^*}(z) )^{-2}
     \sum_{q,s,m=1}^{\infty} 
             \bigg[
             (q-m)(c_m^{(p)})^2c_q^{(p)}c_s^{(p)}
       \epsilon_{qs}    
             \bigg]
             z^{q+m-1} \bar{z}^{s+m}.
  \end{equation}
 Observe that the coefficient inside $[\ldots]$ in the above sum
 vanishes if $q=m$. 
This allows to separate the above sum into 
$$\sum_{q=1, s\geq1, m\geq2}\;\;
\text{and}\;\; \sum_{q\geq 2, s\geq1, m\geq1}.$$
We first tackle the sum over 
 $q=1$, $s\geq1$ and $m\geq2$, 
 focusing on the cases $s, m\leq \delta'_p$; 
 then we deal with the sum over $q\geq2$, $s\geq1$ and $m\geq1$, 
 focusing on $q, s,m\leq \delta'_p$, 
 before we also address the cases of ``large indices''
 ($\max\{q, s, m\}\geq\delta'_p+1$). 
 
  ~
 
 \noindent
 \textit{Head terms. --- }
 We look at first 
  \begin{equation}   \label{e:hdsum}
  I_{p,\delta'_p}(z)  = \sum_{s=1}^{\delta'_p}
  \sum_{m=2}^{\delta'_p} 
      \bigg[
 (1-m)(c_m^{(p)})^2c_1^{(p)}c_s^{(p)}    \epsilon_{1s}    
      \bigg]
         z^{m} \bar{z}^{s+m}.
   \end{equation}
 By  \eqref{e:echelon1}, \eqref{e:echelon2},
 \eqref{e:a_alphaell} and \eqref{eq:4.7a},  
uniformly for $q,s\in  \{1,\ldots,\delta'_p\}$,
  \begin{equation}\label{eq:4.17a}
   \epsilon_{qs} 
= \sum_{\ell=1}^{{\rm min}\{q, s\}} \frac{a_{q\ell}^{(p)}}{c_q^{(p)}}
        \frac{\overline{a_{s\ell}^{(p)}}}{c_s^{(p)}} 
                   - \delta_{qs}
=  \mathcal{O}( \delta'_p p\,  e^{-\kappa p})       .           
  \end{equation}
 For all $t\in\{2,\ldots,\delta'_p\}$, 
 $j\in \{1,\ldots,\delta'_p\}$,  by \eqref{eq:2.7a}, 	
  \begin{equation}\label{eq:4.18a}
 \big| (t-j)c_t^{(p)}\big| 
      \leq \delta'_p \Big(\frac{t}{t-1}\Big)^{(p-1)/2}c_{t-1}^{(p)} 
	  \leq \delta'_p 2^{(p-1)/2}c_{t-1}^{(p)}.
  \end{equation} 
From  \eqref{e:hdsum}, \eqref{eq:4.17a} and \eqref{eq:4.18a}, we get
\begin{align}\label{eq:4.13a}\begin{split}
  \Big|I_{p,\delta'_p}(z)  \Big|   &\leq 
        \sum_{s=1}^{\delta'_p}
        \sum_{m=2}^{\delta'_p}      
 (m-1)(c_m^{(p)})^2c_1^{(p)}c_s^{(p)} |\epsilon_{1s}| |z|^{s+2m}  \\
 &\leq \mathcal{O}\big((\delta'_p)^{2} 
 2^{p/2} p\, e^{-\kappa p}\big)\cdot
\sum_{s=1}^{\delta'_p}
\sum_{m=2}^{\delta'_p} c_m^{(p)}c_{m-1}^{(p)}
c_1^{(p)}c_s^{(p)}  |z|^{s+2m} .
\end{split}  \end{align}
 But 
  \begin{multline}\label{eq:4.14a}
   \sum_{s=1}^{\delta'_p}
\sum_{m=2}^{\delta'_p} c_m^{(p)}c_{m-1}^{(p)}
c_1^{(p)}c_s^{(p)}  |z|^{s+2m}
   = 
 \bigg(\sum_{s=1}^{\delta'_p} c_1^{(p)}c_s^{(p)} |z|^{1+s}\bigg)
\bigg(\sum_{m=2}^{\delta'_p} c_m^{(p)}c_{m-1}^{(p)}
 |z|^{2m-1}\bigg)        \\
    \leq  \frac{1}{2}\bigg(\delta'_p (c_1^{(p)})^2|z|^2 + 
   \sum_{s=1}^{\delta'_p} (c_s^{(p)})^2 |z|^{2s}\bigg)
\bigg(\sum_{m=2}^{\delta'_p} (c_m^{(p)})^2|z|^{2m}\bigg)^{\tfrac12}
 \bigg(\sum_{m=2}^{\delta'_p} (c_{m-1}^{(p)})^2|z|^{2(m-1)}
 \bigg)^{\tfrac12}          \\
 \leq (\delta'_p+1)
\bigg(\sum_{j=1}^{\infty} (c_{j}^{(p)})^2|z|^{2j}\bigg)^2   
=     (\delta'_p+1)( \beta^{\D^*}_p (z))^2 .
      \end{multline}
We proceed similarly with the sum
  \begin{align}\label{eq:4.16a}
  II_{p,\delta'_p}(z)  &= \sum_{q=2}^{\delta'_p}  \sum_{s=1}^{\delta'_p} 
\sum_{m=1}^{\delta'_p}
\Big[(q-m)(c_m^{(p)})^2c_q^{(p)}c_s^{(p)}    \epsilon_{qs}                                       
         \Big]z^{q+m-1}\bar{z}^{s+m}  . 
   \end{align}
We have analogously to \eqref{eq:4.14a},
    \begin{equation}\label{eq:4.19a}
    \begin{split}
    \sum_{q=2}^{\delta'_p} 
            \sum_{s=1}^{\delta'_p} 
            \sum_{m=1}^{\delta'_p}&
               (c_m^{(p)})^2 c_{q-1}^{(p)} c_s^{(p)}|z|^{q+s+2m-1} \\
        &= 
            \bigg(\sum_{q=2}^{\delta'_p} c_{q-1}^{(p)}|z|^{q-1}\bigg)
            \bigg(\sum_{s=1}^{\delta'_p}  c_s^{(p)}|z|^{s} \bigg)
 \bigg(\sum_{m=1}^{\delta'_p} (c_m^{(p)})^2 |z|^{2m}\bigg)   \\
 &\leq \delta'_p
\bigg(\sum_{q=2}^{\delta'_p} (c_{q-1}^{(p)})^2|z|^{2q-2}
\bigg)^{\tfrac12}
\bigg(\sum_{s=1}^{\delta'_p}  (c_s^{(p)})^2|z|^{2s} \bigg)^{\tfrac12}
\bigg(\sum_{m=1}^{\delta'_p} (c_m^{(p)})^2 |z|^{2m}\bigg)    \\
       &\leq \delta'_p ( \beta^{\D^*}_p (z))^2    .
       \end{split}
  \end{equation}
 From \eqref{eq:4.17a}, \eqref{eq:4.18a}, 
 \eqref{eq:4.16a} and \eqref{eq:4.19a}, we get
  \begin{align}\label{eq:4.20a}\begin{split}
   \bigg|  II_{p,\delta'_p}(z)  \bigg|       
       &\leq \sum_{q=2}^{\delta'_p} 
         \sum_{s=1}^{\delta'_p} 
         \sum_{m=1}^{\delta'_p}
           (c_m^{(p)})^2\big|(q-m) c_q^{(p)}\big| c_s^{(p)}|\epsilon_{qs}|
           |z|^{q+s+2m-1}             \\
&\leq\mathcal{O}\big( (\delta'_p)^2 2^{p/2} p\, 
e^{- \kappa p}\big)\cdot
            \sum_{q=2}^{\delta'_p} 
            \sum_{s=1}^{\delta'_p} 
            \sum_{m=1}^{\delta'_p}
               (c_m^{(p)})^2 c_{q-1}^{(p)} c_s^{(p)}|z|^{q+s+2m-1}   \\
&\leq   \mathcal{O}\big((\delta'_p)^{3} 2^{p/2} p\, e^{-\kappa p}\big)
\cdot    ( \beta^{\D^*}_p (z))^2 .         
 \end{split}   \end{align}

 \noindent
 \textit{Tail terms. --- } Set
  \begin{align}\label{eq:4.22a}\begin{split}
   \mathcal{A}_p^1 =& \{(q,s,m)\in (\N^{*})^{3}\,: \, 
                         q\geq \delta'_p+1;\,s,m\leq \delta'_p\},        \\ 
   \mathcal{A}_p^2 =& \{(q,s,m)\in (\N^{*})^{3}\,: \,
                         s\geq \delta'_p+1;\,m\leq \delta'_p\} ,        \\
   \mathcal{A}_p^3 =& \{(q,s,m)\in (\N^{*})^{3}\,: \, 
                         m\geq \delta'_p+1\} \}.    
 \end{split} \end{align}
 For $j=1,2,3$, set 
 \begin{equation}   \label{e:tailsumA1}
  I(\mathcal{A}_p^j)(z) =  \sum_{(q,s,m)\in \mathcal{A}_p^j} 
       (q-m)(c_m^{(p)})^2c_q^{(p)}c_s^{(p)}\epsilon_{qs}
	   z^{q+m-1}\bar{z}^{s+m}.
  \end{equation}
By   \eqref{eq:4.8a}, \eqref{e:hdsum}, \eqref{eq:4.16a}
and \eqref{e:tailsumA1}, we have
 \begin{align}\label{eq:4.25a}
   \frac{d}{dz}\frac{B_p}{B_p^{\D^*}} (z)
 =  (\beta_p^{\D^*}(z) )^{-2}
 \Big( I_{p,\delta'_{p}}(z) +II_{p,\delta'_{p}}(z) + I(\mathcal{A}_p^1)(z) 
 + I(\mathcal{A}_p^2)(z)+  I(\mathcal{A}_p^3)(z)  \Big).
  \end{align}
  We now look at the remaining terms of the sum in 
  \eqref{eq:4.25a}, i.e.,  $I(\mathcal{A}_p^j) (j=1,2,3)$.
  
 First, for all triple $(q,s,m)$ of $\mathcal{A}^1_p$, 
 as $q\geq \delta'_p+1 > \delta'_p \geq s$, 
 by \eqref{eq:3.98a}, \eqref{eq:4.7a}, one has: 
  \begin{align}\label{eq:4.26a}
	  c_q^{(p)}c_s^{(p)}\epsilon_{qs}
     = \sum_{\ell=1}^{d_p} a_{q\ell}^{(p)}
                           \overline{a_{s\ell}^{(p)}}
     = \sum_{\ell=1}^{\delta'_p} a_{q\ell}^{(p)}
                           \overline{a_{s\ell}^{(p)}}.
  \end{align}						   
 %
From \eqref{e:tailsumA1} and \eqref{eq:4.26a},
  we have
  \begin{multline}\label{eq:4.28a}
 | I(\mathcal{A}_p^1)(z)|
  \leq C   d_p 
     \sum_{(q,s,m)\in \mathcal{A}_p^1} 
              q \Big(\sup_{1\leq\ell\leq d_p}|a_{q\ell}^{(p)}|\Big)
              \Big(\sup_{1\leq\ell\leq d_p}|a_{s\ell}^{(p)}|\Big)
              (c_m^{(p)})^2 |z|^{q+s+2m-1}                     \\
= 	C   d_p 	\bigg(\sum_{q=\delta'_p+1}^{\infty}
q   \Big(\sup_{1\leq\ell\leq d_p}|a_{q\ell}^{(p)}|\Big)   |z|^{q-1}
  \bigg) 	\\  			  
       \times  \bigg(\sum_{s=1}^{\delta'_p} 
\Big(\sup_{1\leq\ell\leq d_p}|a_{s\ell}^{(p)}|\Big)|z|^{s}\bigg)
         \bigg(\sum_{m=1}^{\delta'_p}(c_m^{(p)})^2|z|^{2m}\bigg).
  \end{multline}
By  \eqref{eq:3.98a} and \eqref{e:a_alphaell}
we get uniformly in $j\in \{1,\cdots,\delta'_p\}$, 
\begin{align}\label{eq:4.27a}
\sup_{1\leq\ell\leq d_p}|a_{j\ell}^{(p)}|
=\sup_{1\leq\ell\leq \delta'_p}|a_{j\ell}^{(p)}| \leq Cc_j^{(p)}.
 \end{align}	
 By 
 \eqref{eq:4.2a}  
 and \eqref{eq:4.27a}, observe that  for $z\in \D^{*}_{r}$,
  \begin{multline}\label{eq:4.29a}
   \sum_{s=1}^{\delta'_p} 
\Big(\sup_{1\leq\ell\leq d_p}\big|a_{s\ell}^{(p)}\big|\Big)|z|^{s}
    \leq  C \sum_{s=1}^{\delta'_p} c_s^{(p)}|z|^{s}\\
    \leq C(\delta'_p)^{1/2}
          \bigg(\sum_{s=1}^{\delta'_p} 
                  (c_s^{(p)})^2|z|^{2s}\bigg)^{1/2}                 
    \leq C(\delta'_p)^{1/2}  (\beta_p^{\D^*}(z) )^{1/2}.
\end{multline}
Now we give an estimate via $\beta_p^{\D^*}(z)$ for the sum  
$\sum_{q\geq \delta'_p+1}$ in \eqref{eq:4.28a}.
 Recall that for $\xi\in [0,1)$ and $N\geq 0$ we have that 
  \begin{equation}\label{eq:4.35a}
   \sum_{q=N+1}^{\infty}q\xi^{q-1} 
 =\Big(   \sum_{q=N+1}^{\infty}\xi^{q} \Big)'  
 =   \frac{(N+1)\xi^{N}-N\xi^{N+1}}{(1-\xi)^2} 
     \leq \frac{(N+1)\xi^{N}}{(1-\xi)^2}, 
  \end{equation}
 thus, if $|z|\leq  r$,   
   \begin{equation}\label{eq:4.36a}
   \sum_{q=\delta'_p+1}^{\infty}q\Big(\frac{|z|}{2r}\Big)^{q-1}
\leq (\delta'_p+1)\Big(\frac{|z|}{2r}\Big)^{\delta'_p}
\Big(1-\frac{|z|}{2r}\Big)^{- 2}
     \leq 4(\delta'_p+1)  \Big(\frac{|z|}{2r}\Big)^{\delta'_p}.
  \end{equation}

  Taking now 
    \begin{equation}\label{eq:4.38a}
   A' = \frac{1}{2\alpha'}, \quad \alpha'= \frac{c(\kappa)}{4 |\log r|} 
    \quad \text{ and } \quad
   c' = r  e^{1/2\alpha'}\big|\!
   \log\big(|2r|^{2}\big)\big|^{1/2\alpha'},
  \end{equation}
we obtain from \eqref{eq:3.35a}
that  for any $\tau \in \N$ fixed, 
  \begin{equation}\label{eq:4.39a}
\alpha' p\leq \delta'_{p} -\tau \quad \text{ for } \qquad p\gg1.
  \end{equation}  
 Thus, as in \eqref{eq:2.31a},  we have by \eqref{eq:4.38a} 
 and \eqref{eq:4.39a} for $\tau\in \N$ fixed,  
  \begin{equation}\label{eq:4.41a}
\Big(\frac{|z|}{2r}\Big)^{2(\delta'_p-\tau)/p}
\frac{1}{|\log(|2r|^{2})|}
\leq \Big(\frac{|z|}{2r}\Big)^{2\alpha'}
\frac{1}{|\log(|2r|^{2})|}
\leq 2^{- 2\alpha'}\frac{e}{p},
  \end{equation}
for $p\gg1$,  $|z|\leq c'p^{-A'}$.  
To conclude, we estimate by 
\eqref{eq:2.34a}, \eqref{eq:4.2a} and \eqref{eq:4.41a} 
for any $\tau\in \N$ fixed, 
\begin{multline}\label{eq:4.43a}
   \Big|\!\log\Big(|2r|^{2}\Big)\Big|^{-p/2}
   \left(\frac{|z|}{2r}\right)^{\delta'_p-\tau +1}\\[2pt]
       =  \frac{1}{2 r} 
 \bigg( \left(\frac{|z|}{2r}\right)^{2(\delta'_p-\tau)/p}
\frac{1}{ |\log(|2r|^{2})|}\bigg)^{p/2}   
 \Big(2\pi (p-2)!\Big)^{1/2} c_1^{(p)} |z|\\
\leq C p^{-1/2}	2^{- \alpha' p} \beta^{\D^{*}}_{p}(z)^{1/2}	   .    
\end{multline}
for all $p\gg 1$ and $|z|\leq c'p^{-A'}$, 
Thus by  \eqref{eq:2.27a}, 
\eqref{eq:4.36a} and \eqref{eq:4.43a} for $\tau=1$, we have 
 for all $p\gg 1$ and $|z|\leq c'p^{-A'}$, 
\begin{multline}\label{eq:4.42a}
\sum_{q=\delta'_p+1}^{\infty}
q  \Big(\sup_{1\leq\ell\leq d_p}|a_{q\ell}^{(p)}|\Big)   |z|^{q-1}
\leq \frac{Cp^{1/2}}{2r}  \Big|\!\log(|2r|^{2})\Big|^{-p/2}
\sum_{q=\delta'_p+1}^{\infty}q\Big(\frac{|z|}{2r}\Big)^{q-1}\\
\leq C \delta'_p 2^{- \alpha' p}\beta^{\D^{*}}_{p}(z)^{1/2}.
\end{multline} 
By \eqref{eq:2.4a},  \eqref{eq:4.28a}, \eqref{eq:4.29a}
and \eqref{eq:4.42a}
we have for all $p\gg 1$ and $|z|\leq c'p^{-A'}$,  
\begin{align}\label{eq:4.44a}
    | I(\mathcal{A}_p^1)(z)| \leq C  (\delta'_{p})^{3/2}
d_{p}  2^{-\alpha' p}  (\beta_p^{\D^*}(z) )^{2}
\leq C p^{5/2} 2^{-\alpha' p}  (\beta_p^{\D^*}(z) )^{2}.
\end{align} 
 
\noindent
\textit{Sums over $\mathcal{A}^2_p$ and $\mathcal{A}^3_p$. --- }
We continue to work on the estimates of the tail terms.
We first deal with the sum over $\mathcal{A}^2_p$. 
By \eqref{eq:4.7a} and \eqref{e:tailsumA1}, one has: 
\begin{align}\label{eq:4.46a}\begin{split}
I(\mathcal{A}_p^2)(z)        =& \sum_{(q,s,m)\in \mathcal{A}_p^2} 
(q-m)(c_m^{(p)})^2 \bigg(\sum_{\ell=1}^{d_p}a_{q\ell}^{(p)}
\overline{a_{s\ell}^{(p)}}\bigg)  z^{q+m-1}\bar{z}^{s+m}        \\
         &- \sum_{s=\delta'_p+1}^{\infty}\sum_{m=1}^{\delta'_p} 
               (s-m)(c_m^{(p)})^2(c_s^{(p)})^2 z^{s+m-1}\bar{z}^{s+m}.              \\
        =&\!: S_1 - S_2. 
 \end{split}  \end{align}
 Now, since $|q-m|\leq qm$ for all $(q,s,m)\in \mathcal{A}_p^2$, 
we obtain, 
  \begin{multline}   \label{e:S1}
   |S_1| \leq d_p 
               \sum_{q=1}^{\infty} q\Big(\sup_{1\leq\ell\leq d_p}
			   |a_{q\ell}^{(p)}|\Big)|z|^q  \cdot 
               \sum_{s=\delta'_p+1}^{\infty} 
\Big(\sup_{1\leq\ell\leq d_p}|a_{s\ell}^{(p)}|\Big)|z|^{s-1} \cdot
              \sum_{m=1}^{\delta'_p} 
			  m(c_m^{(p)})^2|z|^{2m}.
  \end{multline}
By  \eqref{eq:4.27a}, \eqref{eq:4.29a} 
and \eqref{eq:4.42a}  we get on  $|z|\leq c'p^{-A'}$, 
   \begin{multline}\label{eq:4.48a}
           \sum_{q=1}^{\infty} q\Big(\sup_{1\leq\ell\leq d_p}
		   |a_{q\ell}^{(p)}|\Big)|z|^q
             \leq C\delta'_p\sum_{q=1}^{\delta'_p} c_q^{(p)}|z|^q  
+ C |z|\delta'_p 2^{- \alpha' p}\beta^{\D^{*}}_{p}(z)^{1/2}	\\	 
             = \mathcal{O}\big( (\delta'_{p})^{3/2} + 
\delta'_{p}	2^{- \alpha' p}\Big)
	(\beta_p^{\D^*}(z) )^{1/2}.
\end{multline}
From \eqref{eq:2.27a} and \eqref{eq:4.43a} for $\tau=1$ 
we infer that we have for $|z|\leq c'p^{-A'}$, 
\begin{equation}\label{eq:4.49a}
\begin{split}
          \sum_{s=\delta'_p+1}^{\infty} \Big(\sup_{1\leq\ell\leq d_p}
		  |a_{s\ell}^{(p)}&|\Big)|z|^{s-1}
             \leq Cp^{1/2} 
			  \Big|\!\log\Big(|2r|^{2}\Big)\Big|^{-p/2}
  \sum_{s=\delta'_p+1}^{\infty}\Big(\frac{1}{2r}\Big)^{s}
				  |z|^{s-1}    \\
&= \frac{C}{2r}C p^{1/2}  \Big|\!\log\Big(|2r|^{2}\Big)\Big|^{-p/2}                        
\Big(\frac{|z|}{2r}\Big)^{\delta'_{p}} \frac{1}{1-|z|/2r}                                      \\
&\leq  C 2^{- \alpha' p }	(\beta_p^{\D^*}(z) )^{1/2}.
\end{split}
\end{equation}
        %
 Obviously, 
         \begin{align}\label{eq:4.50a}
\sum_{m=1}^{\delta'_p} m(c_m^{(p)})^2|z|^{2m}
           \leq \delta'_p\sum_{m=1}^{\delta'_p} (c_m^{(p)})^2|z|^{2m}
   \leq   \delta'_p\beta_p^{\D^*}(z).
    \end{align}
 Thus, using these three estimates \eqref{eq:4.48a}--\eqref{eq:4.50a}
together with \eqref{eq:2.4a},  \eqref{eq:3.35a}, we see that
\eqref{e:S1} 
yields: 
  \begin{equation}   \label{e:estS1}
   |S_1| =  \mathcal{O}(p\cdot p^{3/2}\cdot p)
	2^{- \alpha' p}	\beta_p^{\D^*}(z)^{2} \quad
	\text{ for }  |z|\leq c'p^{-A'}.	
		\end{equation}
From \eqref{eq:4.46a},
\begin{align}\label{eq:4.52a}\begin{split}
   |S_2| \leq& \sum_{s=\delta'_p+1}^{\infty}\sum_{m=1}^{\delta'_p} 
 |s-m|(c_m^{(p)})^2(c_s^{(p)})^2 |z|^{2s+2m-1}                \\
 \leq& \bigg(\sum_{s=\delta'_p+1}^{\infty}s(c_s^{(p)})^2
 |z|^{2s-1}\bigg)
  \bigg(\sum_{m=1}^{\delta'_p} (c_m^{(p)})^2|z|^{2m}\bigg).
  \end{split}\end{align}
Note that by the argument in \eqref{eq:2.27a} for $\D^{*}$
  (or directly from  \eqref{eq:2.8a}, \eqref{eq:2.12a}),
  there exists $C>0$ such that for any 
  $s\in \N^*$, $p\geq 2$, we have
  \begin{align}\label{eq:4.54a}
|c_s^{(p)}|\leq 
Cp^{1/2}\big(2r\big)^{-s}|\!\log(|2r|^{2})|^{-p/2}.
 \end{align}
By \eqref{eq:3.35a}, \eqref{eq:4.35a}, \eqref{eq:4.43a} for $\tau=1$,  
and \eqref{eq:4.54a}, we get as in \eqref{eq:4.36a}
for all $p\gg 1$ and $|z|\leq c'p^{-A'}$, 
    \begin{multline}\label{eq:4.55a}
\bigg(\sum_{s=\delta'_p+1}^{\infty}s(c_s^{(p)})^2
 |z|^{2s-1}\bigg) \leq C\Big|\!\log\Big(|2r|^{2}\Big)\Big|^{-p}
   \frac{|z|}{4r^{2}}  p
\bigg(\sum_{s=\delta'_p+1}^{\infty} s \Big(\frac{|z|}{2 r}\Big)^{2s-2}
\bigg)                       \\
         \leq  Cp\, \Big|\!\log\Big(|2r|^{2}\Big)\Big|^{-p}|z|
               \frac{(\delta'_p+1)}{(1-(|z|/2r)^{2})^2}
   \Big(\frac{|z|}{2 r}\Big)^{2\delta'_p}  
   \leq Cp   2^{- 2\alpha' p}	\beta_p^{\D^*}(z). 
  \end{multline}
  By \eqref{eq:4.46a}, \eqref{e:estS1},
 \eqref{eq:4.52a} and \eqref{eq:4.55a} we obtain
  \begin{equation}\label{eq:4.60a}
| I(\mathcal{A}_p^2)(z) |
 = \mathcal{O}( p^{7/2}  2^{- \alpha' p})	\beta_p^{\D^*}(z)^{2}
 \quad \text{ on  } |z|\leq c'p^{-A'}.
  \end{equation}
 We finally deal with the sum over $\mathcal{A}_p^3$, 
using the same principles\footnote{
  Fine uniform control for small indices, 
  rough control via Cauchy formula for large indices, 
  sacrifice of a few powers of $|z|$ and restriction to $|z|\leq cp^{-A}$ 
  for resulting sums.}. 
 Write: 
   \begin{align}\label{eq:4.61a}\begin{split}
 I(\mathcal{A}_p^3)(z)
 =& \sum_{(q,s,m)\in \mathcal{A}_p^3}  (q-m)(c_m^{(p)})^2
\bigg(\sum_{\ell=1}^{d_p}a_{q\ell}^{(p)}\overline{a_{s\ell}^{(p)}}\bigg)
              z^{q+m-1}\bar{z}^{s+m}                 \\
         &- \sum_{s=1}^{\infty}\sum_{m=\delta'_p+1}^{\infty} 
(s-m)(c_m^{(p)})^2(c_s^{(p)})^2 z^{s+m-1}\bar{z}^{s+m}\\
        =&\!: S_1' - S_2'. 
 \end{split}  \end{align} 
 On the one hand, rather similarly as for \eqref{e:S1} 
 (observe the precise exponents though), 
  \begin{multline}   \label{e:S1'}
   |S_1'| \leq d_p 
               \sum_{q=1}^{\infty} q\Big(\sup_{1\leq\ell\leq d_p}
			   |a_{q\ell}^{(p)}|\Big)|z|^q
\sum_{s=1}^{\infty} 
			   \Big(\sup_{1\leq\ell\leq d_p}|a_{s\ell}^{(p)}|\Big)|z|^{s}
\sum_{m=\delta'_p+1}^{\infty}  m(c_m^{(p)})^2|z|^{2m-1}.
  \end{multline}
Again, we deal separately with $\sum_{s=1}^{\delta'_p}$ 
		and $\sum_{s=\delta'_p+1}^{+\infty}$
from \eqref{eq:4.29a}, \eqref{eq:4.49a}.
        %
 In conclusion, by \eqref{eq:2.4a}, \eqref{eq:4.29a},
 \eqref{eq:4.48a}, \eqref{eq:4.49a}, \eqref{eq:4.55a}
and \eqref{e:S1'}, 
we have on $|z|\leq c'p^{-A'}$,
  \begin{equation}\label{eq:4.64a} 
   |S'_1| \leq \mathcal{O}(p^{4} 2^{- 2 \alpha' p})	
   \beta_p^{\D^*}(z)^{2}.
  \end{equation}
 On the other hand,  we have by \eqref{eq:4.55a} on the set 
 $|z|\leq c'p^{-A'}$,
  \begin{align}\label{eq:4.65a} \begin{split}
   |S_2'| &\leq \sum_{q=1}^{\infty}\sum_{m=\delta'_p+1}^{\infty} 
  |q-m|(c_m^{(p)})^2(c_q^{(p)})^2 |z|^{2q+2m-1}   \\
&\leq \bigg(\sum_{q=1}^{\infty}q(c_q^{(p)})^2|z|^{2q}\bigg)
 \bigg(\sum_{m=\delta'_p+1}^{\infty}m(c_m^{(p)})^2|z|^{2m-1}\bigg)\\ 
 &\leq C \Big( \delta'_p \sum_{q=1}^{\delta'_{p}}(c_q^{(p)})^2|z|^{2q}
 +p   2^{- 2\alpha' p}	
   \beta_p^{\D^*}(z)\Big)p  2^{- 2\alpha' p}	
   \beta_p^{\D^*}(z)\\
   &\leq C  2^{- 2\alpha' p}	 p^{2}\,
   \beta_p^{\D^*}(z)^{2}.
 \end{split} \end{align}
By \eqref{eq:4.61a}, \eqref{eq:4.64a} and \eqref{eq:4.65a}
we have on the set $|z|\leq c'p^{-A'}$,
\begin{equation}\label{eq:4.66a}
   \bigg| I(\mathcal{A}_p^3)(z)   \bigg|
   = \mathcal{O}(p^{4}   2^{-2\alpha' p}  )\beta_p^{\D^*}(z)^{2}.
  \end{equation}
 \noindent
 \textit{Conclusion. --- }
  We sum up the estimates above (head terms \eqref{eq:4.13a}, 
  \eqref{eq:4.14a}, \eqref{eq:4.20a}, and tail terms \eqref{eq:4.44a},
  \eqref{eq:4.60a}, \eqref{eq:4.66a}) in \eqref{eq:4.25a}, 
  with $\kappa$ any fixed number larger than $\frac12 \log2$,
  and obtain for some $\gamma>0$,
   \begin{equation}\label{eq:4.91a}
    \sup_{|z|\leq c'p^{-A'}} 
       \bigg|\frac{d}{dz}\bigg(\frac{B_p}{B_p^{\D^*}}\bigg)(z)\bigg|
       = \mathcal{O}(e^{-\gamma p}). 
   \end{equation}
    Applying Theorem \ref{thm_MainThm} for $k=1,\delta=0$, we get
  \begin{equation}\label{eq:4.92a}
   \sup_{c'p^{-A'}\leq |z| \leq r}
     |z|\left|\log(|z|^2)\right|
     \left|\frac{d}{dz}(B_p-B_p^{\D^*})(z)\right|
     =\mathcal{O}(p^{-\infty}),
 \end{equation}	 
which can be rephrased as follows:
 \begin{equation}\label{eq:4.93a}
\sup_{c'p^{-A'}\leq |z| \leq r}
    \left|\frac{d}{dz}(B_p -B_p^{\D^*})(z)\right|
    = \mathcal{O}(p^{-\infty}). 
	 \end{equation}	 
Estimates \eqref{eq:2.12a}, \eqref{eq:4.91a}, \eqref{eq:4.93a} 
yield \eqref{e:diffquot} for $k=1$.

 Higher $k$-order estimates are established along the same lines: 
 (1) the sum over the set of indices in $\mathcal{A}^{j}_{p}$ 
 where one of indices satisfies $\geq \delta'_{p}+1$,
 will be controlled by a polynomial in $p$ times 
 $2^{-\alpha' p}\beta_p^{\D^*}(z)^{k}$;
 (2) to handle the sum over 
 the set of indices $\leq \delta'_{p}$, we observe 
 first that the contribution from the terms with sum of indices 
 $<2k+2$ is zero, so 
we will increase $\kappa$ to absorb the exponential factor
in the estimates. Thus the analogue of \eqref{eq:4.91a} holds
for $k>1$. 
We exemplify this for the second derivative $\frac{d^{2}}{dz^{2}}$ 
to show how the above argument 
works. From \eqref{eq:4.8a}, we get
  \begin{multline}\label{eq:4.95a}
 \frac{d^{2}}{dz^{2}}\frac{B_p}{B_p^{\D^*}} (z)
 =  (\beta_p^{\D^*}(z) )^{-3}     \\
\times\sum_{q,s,t,m=1}^{\infty} 
             (q-m)(q+m-1-2t)(c_m^{(p)})^2(c_t^{(p)})^2c_q^{(p)}c_s^{(p)}
       \epsilon_{qs}    
             z^{q+m-2+t} \bar{z}^{s+m+t}.
  \end{multline}
It is clear that the contribution of the indices with
$q+m+t< 5$ is zero, so the trick 
\eqref{eq:4.18a} works even in the presence of a $z^{-2}$-term 
in \eqref{eq:4.95a}.
   \hfill \cqfd
  
  \section{Applications} \label{eq:s4}
  
 %
Theorem \ref{thm:diffquot} can be interpreted in terms of 
Kodaira embeddings.
Following the seminal papers \cite{bou,Ca99,DLM06,Don,MM08,Ti90,Z98}
one of the main applications of the expansion of the Bergman kernel
is the convergence of the induced Fubini-Study metrics by Kodaira maps.
Let us consider the Kodaira map at level $p\geq2$ induced by 
$H^{0}_{(2)}(\Sigma, L^{p})$, which is a meromorphic map
defined by
\begin{equation}\label{e:kod1} 
\jmath_{p,(2)}:\Sigma\dashrightarrow 
\mathbb{P}(H^{0}_{(2)}(\Sigma, L^{p})^{*})
\cong\C\mathbb{P}^{d_p-1}\,,\:\: x\longmapsto 
\big\{\sigma\in H^{0}_{(2)}(\Sigma, L^{p}):\sigma(x)=0\big\}.
\end{equation}
Recall that by \cite[Remark 3.2]{bkp}
the sections of $H^0_{(2)}(\Sigma,L^p)$
extend to holomorphic sections of $L^p$ over $\overline\Sigma$
that vanish at the punctures and this gives an
identification 
\begin{equation}\label{e:bs22}
H^0_{(2)}(\Sigma,L^p)\cong\{\sigma\in H^0(\overline\Sigma,L^p): 
\sigma|_D=0\}.
\end{equation}
Let $\sigma_D$ be the canonical
section of the bundle $\mathscr{O}_{\overline{\Sigma}}(D)$.
The map 
\begin{equation}\label{e:h2h0}
H^0(\overline\Sigma,L^p\otimes\cO_{\overline{\Sigma}}(-D))
\to \{\sigma\in H^0(\overline\Sigma,L^p): 
\sigma|_D=0\},\quad s\mapsto s\otimes\sigma_D\,,
\end{equation}
is an isomorphism and we have an identification
$H^0(\overline\Sigma,L^p\otimes\cO_{\overline{\Sigma}}(-D))
\otimes\sigma_D\cong H^0_{(2)}(\Sigma,L^p)
\subset H^0(\overline\Sigma,L^p)$.
Since the zero divisor of $\sigma_D$ is $D$ we have for $x\in\Sigma$,
\begin{equation}\label{e:hyp}
\big\{\sigma\in H^{0}_{(2)}(\Sigma, L^{p}):\sigma(x)=0\big\}=
\big\{s\in H^0(\overline\Sigma,L^p\otimes\cO_{\overline{\Sigma}}(-D)):
s(x)=0\big\}\otimes\sigma_D.
\end{equation}
Let $\jmath_p$ the Kodaira map defined by 
$H^0(\overline\Sigma,L^p\otimes\cO_{\overline{\Sigma}}(-D))$.
We have by \eqref{e:hyp} the commutative diagram
\begin{equation}\label{e:comkod}
    \xymatrix@C=3pc  {
    {\,\,}\Sigma_{} 
    \ar[rr]^{\jmath_{p,(2)}\qquad\,\,}
    \ar@{^{(}->}[d]
    && \mathbb{P}(H^{0}_{(2)}(\Sigma, L^{p})^{*}) 
    \ar[d]^{\operatorname{Id}} \\
    \overline\Sigma \ar[rr]_{\jmath_{p}\qquad\,\,}
    && \mathbb{P}(H^0(\overline\Sigma,L^p\otimes
    \cO_{\overline{\Sigma}}(-D))^{*})
  }
    \end{equation}
It is well known that $\jmath_{p}$ is a holomorphic embedding for
$p$ large enough, namely for all $p$ satisfying $p\deg (L)-N>2g$
(see \cite[p.\ 215]{GH}). 
Thus $\jmath_{p,(2)}$ is also an embedding for $p$ large enough,
as the restriction of an embedding of $\overline\Sigma$. 

The $L^2$-metric \eqref{eq:2.1a} on $H^{0}_{(2)}(\Sigma, L^{p})$
induces a Fubini-Study K\"ahler metric 
$\omega_{{\rm FS},p}$ on the projective space
$\mathbb{P}(H^{0}_{(2)}(\Sigma, L^{p})^{*})$ and a Fubini-Study 
Hermitian metric $h_{{\rm FS},p}$ on the hyperplane 
line bundle $\cO(1)\to\mathbb{P}(H^{0}_{(2)}(\Sigma, L^{p})^{*})$.
By \cite[Theorem 5.1.6]{mm} $\jmath_{p}$ and $\jmath_{p,(2)}$
induce canonical isomorphisms 
\begin{equation}\label{e:cipk}
\jmath_{p}^*\cO(1)\simeq L^p\otimes\cO(-D)\,,\quad
\jmath_{p,(2)}^*\cO(1)\simeq L^p\big|_{\Sigma}\,. 
\end{equation}
Let $\jmath_{p,(2)}^{*}h_{{\rm FS},p}$ be the 
Hermitian metric induced by $h_{{\rm FS},p}$
via the isomorphism \eqref{e:cipk} on $L^p\big|_{\Sigma}$.

\begin{thm}   \label{thm:kodaira1}
Let $(\Sigma,\omega_{\Sigma}, L, h)$ fulfill conditions
$(\alpha)$ and $(\beta)$. Then as $p\to\infty$,
\begin{equation}\label{e:FS1}
\begin{split}
&\jmath_{p,(2)}^{*}h_{{\rm FS},p}=
\big(1+\mathcal{O}(p^{-\infty})\big)
(B_p^{\D^*})^{-1}h^p\,,\\[2pt]
&\frac1p\jmath_{p,(2)}^{*}\omega_{{\rm FS},p}=
\frac{1}{2\pi }\omega_{\Sigma}+\frac{i}{2\pi p}\partial\overline{\partial}
\log\big(B_p^{\D^*}\big)
+\mathcal{O}(p^{-\infty})\,,
\end{split}
\end{equation}
uniformly on $V_1\cup V_2\cup\ldots\cup V_N$.
\end{thm}
\begin{proof}
We have indeed by \cite[Theorem 5.1.6]{mm},
\begin{equation}\label{e:FS2}
\jmath_{p,(2)}^{*}h_{{\rm FS},p}=
(B_p)^{-1}h^p\,,\quad
\frac1p\jmath_{p,(2)}^{*}\omega_{{\rm FS},p}=
\frac{i}{2\pi } R^{L}+\frac{i}{2\pi p}\partial\overline{\partial}
\log(B_p)\,,
\end{equation}
 so \eqref{e:FS1} follows from
Theorems \ref{thm_apdx} and \ref{thm:diffquot}.
\end{proof}

We compare next the induced Fubini-Study
metrics by $\jmath_{p,(2)}$ on $\Sigma$
and on $\D^*$, and show that they 
differ from each other (modulo the usual identification on
$\D^*_{4r}$ in \eqref{eq:1.6a}
with the neighbourhood of a singularity of $\Sigma$) 
by a sequence of $(1,1)$-forms 
\textit{which is $\mathcal{O}(p^{-\infty})$ (at every order) 
with respect to any smooth reference metric on $\D_r$}: 
the situation is just as good as in the smooth setting. 

The infinite dimensional projective space $\C\mathbb{P}^{\infty}$
is a Hilbert manifold modeled on the space $\ell^2$ of
square-summable sequences of complex numbers $(a_j)_{j\in\N}$
endowed with the norm $\|(a_j)\|=
\big(\sum_{j\geq0}|a_j|^2\big)^{1/2}$.
Then $\C\mathbb{P}^{\infty}=\ell^2\setminus\{0\}/\C^*$
and for $a\in\ell^2$ we denote by $[a]$ its class in 
$\C\mathbb{P}^{\infty}$. 
The affine charts are defined as usual by $U_j=\{[a]:a_j\neq0\}$.
The Fubini-Study metric $\omega_{{\rm FS},\infty}$ is defined
by $\omega_{{\rm FS},\infty}=
\frac{i}{2\pi}\partial\overline\partial\log\|a\|^2$ to the effect that
for a holomorphic map $F:M\to\C\mathbb{P}^{\infty}$ from a
complex manifold $M$ to $\C\mathbb{P}^{\infty}$ we have
$F^*\omega_{{\rm FS},\infty}=
\frac{i}{2\pi}\partial\overline\partial\log\|F\|^2$.
We define the Kodaira map of level $p$ associated with 
$(\D^*, \omega_{\D^*}, \C, h_{\D^*})$ by
using the orthonormal basis \eqref{eq:2.7a} of $H_{(2)}^p(\D^*)$,
\begin{equation}\label{e:kod2} 
\imath_p:\D^* \to\C\mathbb{P}^{\infty}\,,\quad
\imath_p(z)=[c_1^{(p)}z,c_2^{(p)}z^2,\ldots,
c_\ell^{(p)}z^\ell,\ldots]\in\C\mathbb{P}^{\infty},\:\:
z\in\D^*.
\end{equation}
\begin{thm}   \label{thm:kodaira2}
Suppose that $\D^*_{4r}$ 
and $V\subset \Sigma$ 
are identified as in \eqref{eq:1.6a}.
On $\D^*_{4r}$ we set 
\begin{equation}\label{e:etap} 
\imath_p^*\omega_{{\rm FS},\infty} 
- \jmath_{p,(2)}^*\omega_{{\rm FS},p} = \eta_p\,idz\wedge d\bar{z}.
\end{equation}
Then $\eta_p$ extends smoothly to $\D_r$ and
one has, for all $k\geq 0$,  $\ell\geq 0$,
\begin{equation}
 \|\eta_p\|_{C^{k}(\D_r)} \leq C_{k,\ell}\, p^{-\ell}\,,
\qquad\text{as }p\to\infty,
 \end{equation}
 where $ \|\cdot\|_{C^{k}(\D_r)}$ is the  usual $C^{k}$-norm
 on $\D_{r}$.
\end{thm}
\begin{proof}
We first observe that $\imath_p$ is an embedding,
since already $z\mapsto[c_1^{(p)}z,c_2^{(p)}z^2]\in\C\mathbb{P}^{1}$
is an embedding.
We have 
\begin{equation}\label{e:FS3}
\frac{p}{2\pi }\, \omega_{\D^*} = \imath_p^*\omega_{{\rm FS},\infty}
- \frac{i}{2\pi}\partial\overline{\partial}\log\big(B_p^{\D^*}\big),
\end{equation}
and consequently on $\D_{r}^{*}$, 
\begin{equation}\label{e:FS4}
\imath_p^*\omega_{{\rm FS},\infty}-
\jmath_{p,(2)}^*\omega_{{\rm FS},p}=
\frac{i}{2\pi}\partial\overline{\partial}\log\big(B_p^{\D^*}/B_p\big)\,,
\end{equation} 
so the assertion follows from Theorem \ref{thm:diffquot}.
\end{proof}
We finish with an application to random K\"ahler geometry, 
more precisely to the distribution of zeros
of random holomorphic sections \cite{CM11,DS06}.

Let us endow the space $H^{0}_{(2)}(\Sigma, L^{p})$
with a Gaussian probability measure $\mu_p$
induced by the unitary map $H^{0}_{(2)}(\Sigma, L^{p})\cong\C^{d_p}$
given by the choice of an orthonormal basis $(S_j^p)_{j=1}^{d_p}$.
Given a section $s\in H^{0}_{(2)}(\Sigma, L^{p})\subset 
H^{0}(\overline\Sigma, L^{p})$ we denote by $[s=0]$
the zero distribution on $\overline\Sigma$ defined by the zero divisor 
of $s$ on $\overline\Sigma$.
If the zero divisor of $s$ is given by $\sum m_jP_j$, where $m_j\in\N$ and
$P_j\in\overline\Sigma$, then $[s=0]=\sum m_j\delta_{P_j}$,
where $\delta_{P}$ is the delta distribution at $P\in\overline\Sigma$.
We denote by $\langle\LargerCdot,\LargerCdot\rangle$ 
the duality between distributions and test functions.
For a test function $\Phi\in C^\infty(\overline\Sigma)$ and $s$ as above
we have $\langle[s=0],\Phi\rangle=\sum m_j\Phi(P_j)$.

The expectation distribution $\E[s_p=0]$ of the distribution-valued 
random variable 
$H^{0}_{(2)}(\Sigma, L^{p})\ni s_p\mapsto[s_p=0]$ 
is defined by
\begin{equation}\label{e:FS7}
\big\langle \E[s_p=0],\Phi\big\rangle=
\int\limits_{H^{0}_{(2)}(\Sigma, L^{p})}\big\langle[s_p=0],
\Phi\big\rangle\,d\mu_p(s_p),
\end{equation} 
where $\Phi$ is a test function on $\overline\Sigma$. 
We consider the product probability space
\[(\mathcal{H},\mu)=
\left(\prod_{p=1}^\infty H^{0}_{(2)}(\Sigma, L^{p}),
\prod_{p=1}^\infty\mu_p\right). 
\]
\begin{thm}   \label{thm:equi}
(i) The smooth $(1,1)$-form 
$\jmath_{p,(2)}^{*}\omega_{{\rm FS},p}$ extends to
a closed positive $(1,1)$-current on $\overline\Sigma$ denoted
$\gamma_p$ (called Fubini-Study current) and we have
$\E[s_p=0]=\gamma_p$\,.
\\[2pt]
(ii) We have $\frac1p\gamma_p\to \frac{i}{2\pi}R^L$ as $p\to\infty$, weakly
in the sense of currents on $\overline\Sigma$, where $R^L$
is the curvature current of the singular holomorphic bundle $(L,h)$
on $\overline\Sigma$.
\\[2pt]
(iii) For almost all sequences $(s_p)\in(\mathcal{H},\mu)$ we have
$\frac1p[s_p=0]\to \frac{i}{2\pi}R^L$ as $p\to\infty$, weakly
in the sense of currents on $\overline\Sigma$.
\end{thm}
\begin{proof}
The convergence
of the Fubini-Study currents $\gamma_p$ follows 
from \eqref{e:FS1}.
The rest of the assertions follow from the general arguments of 
\cite[Theorems 1.1, 4.3]{CM11}.
The conditions (A)-(C) in \cite[Theorems 1.1, 4.3]{CM11}
are implied by our hypotheses $(\alpha)$, $(\beta)$
and the required local uniform convergence $\frac1p\log B_p\to0$
as $p\to\infty$
on $\Sigma$ is a consequence of \cite[Theorem 6.1.1]{mm}. 
\end{proof}
\begin{small}

\renewcommand{\refname}
 
\end{small}

\end{document}